\documentclass[11pt,reqno]{amsart}

\usepackage[utf8]{inputenc}
\usepackage[margin=1.25in]{geometry}
\usepackage{hyperref}
\usepackage{appendix}
\usepackage{amsfonts}
\usepackage{amsthm}
\usepackage{amssymb}
\usepackage{stmaryrd} 
\usepackage{enumitem}
\usepackage{amsmath}
\usepackage{amsthm}
\usepackage[dvipsnames]{xcolor}
\usepackage{mathrsfs}
\usepackage{lipsum} 

\theoremstyle{plain}
\newtheorem{theorem}{Theorem}[section]
\newtheorem{corollary}[theorem]{Corollary}
\newtheorem{lemma}[theorem]{Lemma}
\newtheorem{Proposition}[theorem]{Proposition}

\newtheorem{Assumption}[theorem]{Assumption}

\newtheorem{Example}[theorem]{Example}
\newtheorem{Definition}[theorem]{Definition}

\newtheorem{fact}[theorem]{Fact}

\theoremstyle{remark}

\newtheorem{remark}[theorem]{Remark}

\usepackage[maxbibnames=99]{biblatex} 
\numberwithin{equation}{section}
\title[Glauber dynamics for RFIM]{Glauber dynamics for random field Ising models on bounded degree graphs and MLSI}

\author{Yi HAN}
\address{Institute for Advanced Study, 1 Einstein Drive, Princeton, NJ
}
\email{hanyi@ias.edu}

\begin{document}

\begin{abstract}
We study the ferromagnetic random field Ising model (RFIM) on a graph $G=(V,E)$ having maximal degree $\Delta$, where the external field at each vertex is an i.i.d. random variable. When the random field distribution is sufficiently anti-concentrated, we prove that with high probability over the quenched randomness of the external field, the Glauber dynamics of this RFIM mixes in polynomial time as a consequence of a Poincaré inequality. This model is relevant to the Griffiths phase where the correlations decay exponentially fast in expectation over the quenched random field,
but contraction does not hold point-wise due to the existence of weak fields that lead to low-temperature behavior. Previously, fast mixing of Glauber dynamics under large disorder was only proven on the integer lattice, and for RFIM on general graphs, only a sampling algorithm based on self-avoiding walks was known. Under a further technical condition that the random fields are bounded, we prove a modified log-Sobolev inequality for the Glauber dynamics. When the random field is weaker but still satisfies weak spatial
mixing (exponential decay of correlations from boundary to bulk)
in expectation, and the graph has
$\alpha$-stretched-exponential growth for some $\alpha<1$,
we prove a weak Poincaré inequality.
If weak spatial mixing holds uniformly along a given sequence
of growing induced subgraphs, we further obtain a polynomial-time
sampling algorithm based on Glauber dynamics with warm starts.
These results extend the corresponding weak Poincaré and sampling
results from the integer lattice to graphs with
$\alpha$-stretched-exponential growth.

\textbf{Note}: an extended abstract of this paper, which covers only polynomial time mixing of Glauber dynamics (Theorem 1.2) but not the MLSI (Theorem 1.5) and the weak spatial mixing results (Theorem 1.10), is recently accepted in \textbf{RANDOM 2026}.
\end{abstract}

\maketitle

\section{Introduction}

Consider a finite graph $G=(V,E)$, the random field Ising model (RFIM) on $G$ at inverse temperature $\beta>0$ is defined as the following random probability measure on $\{-1,+1\}^V$:
\begin{equation}\label{firstrfimmeasure}
    \mu_G(\sigma)\propto \exp\{H(\sigma)\},\quad\sigma\in\{-1,+1\}^V,
\end{equation}
where $H$ is the following Hamiltonian
\begin{equation}
    H(\sigma)=\beta\sum_{(u,v)\in E}\sigma_u\sigma_v+\sum_{u\in V}h_u\sigma_u.
\end{equation}
Here $h=(h_u)_{u\in V}$ are i.i.d. random variables. The vector $h$ is called the external field, serving as a quenched source of disorder in this model. 

There has been much recent attention on RFIM on (subgraphs of) the integer lattice $\mathbb{Z}^d$ in the statistical physics literature. In $d=2$, the existence of an arbitrarily small random field is predicted by Imry and Ma \cite{imry1975random} to suppress the low temperature phase of the zero field ($h=0$) Ising model, so that uniqueness and decay of correlations hold for any $\beta>0$. The prediction was confirmed by Aizenman and Wehr \cite{aizenman1990rounding}, who proved that the expected influence on the spin at the origin by different boundary conditions at distance $r$ away should decay with $r$. For the Gaussian case $h_u\sim \mathcal{N}(0,\sigma^2)$, the breakthrough work \cite{ding2021exponential} confirmed that the influence decays exponentially in $r$ at all $\beta>0$ and $\sigma>0$. This no longer holds for RFIM with dimension $d\geq 3$ and small disorder, where there is a phase transition as $\beta$ varies. For small disorder, there is no decay of influence when $\beta=\infty$ \cite{imbrie1985ground} and when $\beta$ sufficiently large \cite{bricmont1988phase}, and this is recently proven throughout the entire low temperature regime \cite{ding2024long}.

Beyond the integer lattice, the RFIM on a general graph $G$ has also attracted interest from a computational complexity perspective. For a large $\beta>0$ and general external field $h$, estimating the partition function of $\mu_G$ up to a multiplicative error is a \#BIS-hard problem in the worst case \cite{cai2016hardness}, and polynomial-time approximation algorithms are not expected to exist. However, the random field may turn this hard computational problem into a feasible one: for bounded degree graphs $G$ and large enough $\sigma$ where $h_u\sim\mathcal{N}(0,\sigma^2)$, \cite{helmuth2023approximation} constructed a polynomial time approximation algorithm for the partition function of $\mu_G$ over a $1-o(1)$ probability realization of the external fields $h$.

For a general graph $G$, it becomes much less tractable to study the influence decay of a RFIM with small or moderate disorder $h_u$ (that is, smaller $\sigma$ if $h_u\sim\mathcal{N}(0,\sigma^2)$). This is because current RFIM papers on $\mathbb{Z}^d$ that capture the influence decay at $d=2$ or the phase change at $d=3$ use the geometry of $\mathbb{Z}^d$ in a fundamental way and such a delicate description is hard to expect on a general graph $G$. As a natural next step, in this paper we study effective sampling algorithms for the RFIM $\mu_G$ on a general graph $G$. For sufficiently large disorder (say, larger $\sigma$), it is straightforward to verify the decay of correlations for this RFIM via a percolation argument.

In contrast, even for relatively large disorder, the RFIM still possesses major difficulties from a sampling perspective. We consider the Glauber dynamics for RFIM, where we recall that Glauber dynamics (introduced in \cite{glauber1963time}; see also \cite{Martinelli1999Glauber} for more background) is a discrete time random process $(X_n(v))_{v\in V,n\in \mathbb{N}_+}$ reversible w.r.t. $\mu_G$, that at each step $n$ selects a vertex $v\in V$ uniformly with probability $\frac{1}{|V|}$ and then resamples $X_n(v)$ according to the conditional distribution 
\begin{equation}
X_n(v)\sim\mu_G(\sigma_v\in\cdot\mid (\sigma_w)_{w\neq v}=(X_{n-1}(w))_{w\neq v}), 
\end{equation}
and where we set $X_n(w)=X_{n-1}(w)$ for all $w\neq v$.
If we assume that all the $|h_u|$ are uniformly larger than some fixed constant, then it is straightforward to show that the Glauber dynamics for RFIM mixes in time $O(n\log n)$, as in \cite{helmuth2023approximation}, Theorem 7. However, in the typical RFIM setting, although most external fields $h_u$ are large, there are large regions where the absolute values of external fields $|h_u|$ are small, and the Glauber dynamics does not contract at these vertices (see \cite{helmuth2023approximation}, Remark 7). To avoid this problem, \cite{helmuth2023approximation} constructed a polynomial sampling algorithm for RFIM $\mu_G$ on general bounded degree graphs $G$ utilizing the Weitz self-avoiding walk tree \cite{weitz2006counting}, which is a global algorithm that differs in spirit from the Glauber update rule.  As a future direction, \cite{helmuth2023approximation} asked whether one can use Markov chain based algorithms (such as Glauber dynamics) to effectively sample from RFIM.

For Glauber dynamics of RFIM on $\mathbb{Z}^d,d
\geq 2$, \cite{el2026fast} conducted a detailed study for the mixing time of $\mu_G$ and designed a sampling algorithm at weak disorder. Under weak spatial mixing in expectation, they proved a weak
Poincaré inequality for Glauber dynamics and obtained
polynomial-time mixing from a warm start.
Assuming that weak spatial mixing holds uniformly along a suitable
increasing sequence of domains, they further constructed
a polynomial-time sampling algorithm. Under a stronger notion of influence decay and in particular when random fields have large disorder ($h_u$ have good anti-concentration), they prove that Glauber dynamics for $\mu_G$ mixes in polynomial time as a corollary of a full Poincaré inequality. However, the proof in \cite{el2026fast} is highly dependent on the lattice structure of $\mathbb{Z}^d$ due to a coarse graining argument in Section 5, and in Section 4 also somewhat relies on a volume growth exponent on $G$ that rules out graphs of exponential volume growth. Therefore, the Glauber dynamics of $\mu_G$ on general graphs $G$ had remained open.

The main result of this paper is a Poincaré inequality for $\mu_G$ which implies polynomial time mixing for $\mu_G$.

Let $\mathcal{E}_G$ denote the Dirichlet form of the (discrete time) Glauber dynamics, so that for any test function $\varphi:\{-1,+1\}^V\to\mathbb{R}$,

\begin{equation}
\mathcal{E}_{G}(\varphi,\varphi):=\frac{1}{2|V|}\sum_{\sigma\sim\sigma'}\frac{\mu_G(\sigma)\mu_G(\sigma')}{\mu_G(\sigma)+\mu_G(\sigma')}
(\varphi(\sigma)-\varphi(\sigma'))^2,
\end{equation}
so that the summation is taken over pairs $\sigma,\sigma'\in\{-1,1\}^V$ differing only at one coordinate.
Then we define the spectral gap of the chain via 
\begin{equation}
    \mathbf{gap}_G:=\inf_\varphi \frac{\mathcal{E}_G(\varphi,\varphi)}{\operatorname{Var}_G(\varphi)}
\end{equation}
with the infimum taken over all test functions $\varphi:\{-1,1\}^V\to\mathbb{R}$ with non-zero variance.

\subsection{Polynomial time mixing for Glauber dynamics of RFIM with large disorder}

Before stating the main technical result, we outline the anti-concentration assumptions on the external field $h$ in the following:
\begin{Assumption}\label{primaryassumptionpaper}
Let $\Delta\geq 3$ and $p_0\in(0,1)$ be a constant such that $p_0(\Delta-1)<1$, let $K=K(p_0)>0$ be a constant such that
\begin{equation}
    \rho=\rho(\beta,\Delta,K):=\frac{e^{\Delta\beta-K}}{e^{\Delta\beta-K}+e^{-\Delta\beta+K}},
\end{equation} is less than $\frac{p_0}{4}$. Suppose that the RFIM $\mu_G$ with external field $h$ is such that $(h_u)_{u\in V}$ are i.i.d. and such that 
$$
\mathbb{P}_h(|h_u|\leq K)<\frac{p_0}{2}.
$$We then denote by 
$$\alpha_*=\frac{1}{2}\log\frac{(\Delta-2)^{\Delta-2}}{p_0(\Delta-1)^{\Delta-1}(1-p_0)^{\Delta-2}}.$$
Note that for fixed $\Delta>0$, $\alpha_*\uparrow\infty$ as $p_0\downarrow 0$.

\end{Assumption}    

For a graph $G=(V,E)$, we use $|V|$ for the number of vertices in $G$.

Our first main result on Glauber dynamics for RFIM for any bounded degree graph $G$ is as follows:

\begin{theorem}\label{theorem12121}
    Let $G=(V,E)$ be a graph of degree at most $\Delta$. Assume that Assumption \ref{primaryassumptionpaper} is satisfied for some $p_0\in(0,1)$, then with $\mathbb{P}_h$- probability at least $1-O_{\beta,\Delta,p_0}(\frac{1}{|V|})$ over the external field (here $O_{\beta,\Delta,p_0}$ denotes constants that may depend on $\beta$, $\Delta$ and $p_0$), the spectral gap of RFIM measure $\mu_G$ satisfies
    $$
    \mathbf{gap}_G\geq |V|^{-1}\exp\left(-\frac{16\beta\Delta\ln|V|}{\alpha_*}\right),
    $$
    and the $\epsilon$- mixing time of Glauber dynamics for $\mu_G$ is $$t_{mix}(P,\epsilon)=O\left(|V|^{1+\frac{16\beta\Delta}{\alpha_*}}\left(\beta\Delta|V|+\|h\|_1+\log(1/\epsilon)\right)\right).$$
Here $\|h\|_1:=\sum_{x\in V}|h_x|.$
\end{theorem}

\begin{remark} One may also consider the following more general model with different edge weights $\beta_{uv}$.
     Suppose that for each edge $(u,v)\in E$ we assign a constant $\beta_{uv}\in\mathbb{R}$ such that $|\beta_{uv}|\leq\beta$. Consider the following Hamiltonian $\widetilde{H}$ and Ising model $\mu_{\widetilde{G}}$:
    \begin{equation}\label{RFIMFORMIXED}
\widetilde{H}(\sigma)=\sum_{(u,v)\in E}\beta_{uv}\sigma_u\sigma_v+\sum_{u\in V}h_u\sigma_u,\quad \mu_{\widetilde{G}}(\sigma)\approx\exp\{\widetilde{H}(\sigma)\},\quad \sigma\in\{-1,1\}^V.    
\end{equation}
In the ferromagnetic case \(\beta_{uv}\ge 0\) with nonconstant edge weights,
the same strategy should be adaptable by using edge-dependent tilting rates.
We do not pursue this extension here. However, when some $\beta_{uv}<0$, the edge-tilt localization scheme that we will use breaks down, so Theorem \ref{theorem12121} does not readily extend to non-monotone systems. 
\end{remark}

\begin{remark}
    We note that on the integer lattice $\mathbb{Z}^d$, \cite{el2026fast}, Section 5 obtained a mixing time $O(|V|^{1+o(1)})$ for a high probability realization of $h$ via a coarse graining argument (they considered a continuous time Glauber dynamics where vertices are updated at rate 1, and we rewrite their runtime in the discrete setting). We consider much more general graphs (such as expander graphs), and by setting $p_0>0$ sufficiently small (so that $\frac{16\Delta\beta}{\alpha_*}\leq\epsilon$) we can achieve $O(|V|^{2+\epsilon})$ mixing time for any $\epsilon>0$ at fixed total variation accuracy, provided that $\|h\|_1=O(|V|)$. Note that the spectral gap scales like $|V|^{-1-\epsilon}$ when $p_0>0$ is small, which is almost optimal up to the $\epsilon>0$.

     We also note that the anti-concentration condition is designed for a worst case percolation argument to work through, and thus may not be the weakest possible condition. Previous papers concerning a large external field also work under these types of conditions, see \cite{helmuth2023approximation} and \cite{el2026fast}, Section 5.     

\end{remark}

\subsubsection{Modified LSI for RFIM}

For the Glauber dynamics of $\mu_G$ with Markov operator $P$, we define the entropy of $f:\{-1,1\}^V\to\mathbb{R}_+$ via
$$
\operatorname{Ent}_G[f]=\int f(\sigma)\log f(\sigma)d\mu_G(\sigma)-\int f(\sigma)d\mu_G(\sigma)\log(\int f(\sigma)d\mu_G(\sigma))
.$$

Then we define the modified log-Sobolev inequality (MLSI) constant for Glauber dynamics via 
$$
\rho_{LS}(P):=1-\sup_{f:\{-1,1\}^V\to\mathbb{R}_{\geq 0}}\frac{\operatorname{Ent}_G[Pf]}{\operatorname{Ent}_G[f]}.
$$

Classical log-Sobolev and modified log-Sobolev inequalities for spin systems
are well understood in high-temperature or strong-mixing regimes; see, e.g.,
\cite{stroock1992logarithmic,MartinelliOlivieri1994OnePhase,
CaputoMenzTetali2015,chen2021optimal,bauerschmidt2024log}. To the best of our knowledge,
no comparable MLSI bound was previously known for RFIM in the large-disorder
regime beyond the classical high-temperature or entropic-independence setting. As the second main result of this paper, we prove that RFIM on a general graph $G$ satisfies MLSI assuming that the external fields are bounded:

\begin{theorem}\label{mlsitheorembound}
Under the assumptions in Theorem \ref{theorem12121}, assume further that there exists some $M>0$ such that 
\begin{equation}\label{boundexternalfield}
|h_u|\leq M\quad a.s.\quad \forall u\in V.
\end{equation}
Then with $\mathbb{P}_h-$ probability at least $1-O_{\beta,\Delta,p_0}(\frac{1}{|V|}
)$, the Glauber dynamics of $\mu_G$ satisfies a modified log-Sobolev inequality with constant $$
\rho_{LS}(P)\geq \frac{1}{3|V|}\exp\left(-4\beta((C_{\Delta,M,\beta}+1)\frac{4\Delta\ln|V|}{\alpha_*}+1)\right),
$$where $C_{\Delta,M,\beta}$ is a constant defined in \eqref{cemmdelta}.
Consequently, the $\epsilon$-mixing time of Glauber dynamics for $\mu_G$ is
$$
t_{mix}(P,\epsilon)=O(\rho_{LS}(P)^{-1}(\log(\beta\Delta |V|+\|h\|_1)+\log(1/\epsilon))).
$$
\end{theorem}

The MLSI constant of $\mu_G$ is also polynomially small in $|V|$ with a polynomial exponent depending on the constants $M,\Delta,\beta$ and $p_0$.

\begin{remark}
    The bounded external field condition \eqref{boundexternalfield} is likely an artifact of the proof, and we expect that this condition is unnecessary. The condition is used when we apply Lemma \ref{lemma4.20}. Although this condition excludes Gaussian distribution, it still covers the uniform distribution on a long interval $[-M,M]$ whenever $M$ is large enough, and also covers some sufficiently fine atomic discretizations of the uniform distribution.  
\end{remark}

\subsection{Fast relaxation of RFIM under WSM and stretched-exponential growth}

In the final main result, we show that polynomial-time sampling
from the RFIM is possible on graphs with
$\alpha$-stretched-exponential growth for some $\alpha\in(0,1)$,
provided that weak spatial mixing holds uniformly along a given
increasing sequence of induced subgraphs.
The algorithm extends the sequential sampling construction of
\cite{el2026fast}, Theorem 1.4, from the integer lattice
$\mathbb{Z}^d$, $d\geq 2$, to this setting.

The following weak spatial mixing condition characterizes the regime of weak disorder that is already sufficient for fast relaxation:
\begin{Definition}
    Let $G=(V,E)$ be a graph and $\mu_G$ be the RFIM measure on $G$. We say that $\mu_G$ satisfies $\operatorname{WSM}(C)$ for some $C>0$ if for all vertices $v\in V$ and $r\geq 1$, we have 
    $$
\mathbb{E}_h[d_{TV}(\mu_G(\sigma_v\in\cdot\mid \sigma_{\partial B_r(v)}=+),\mu_G(\sigma_v\in\cdot\mid \sigma_{\partial B_r(v)}=-))]\leq Ce^{-r/C}.
    $$Here $\partial B_r(v)$ is the boundary of $B_r(v)$ in $V$ and $B_r(v)$ is the ball of radius $r$ around $v$ with respect to the graph distance on $G$. We use $\sigma_{\partial B_r(v)}=+$ (resp. $\sigma_{\partial B_r(v)}=-$) to denote the all plus (resp. all minus) boundary condition.
\end{Definition}

\begin{remark}
    In the strong disorder regime, for example when Assumption \ref{primaryassumptionpaper} is satisfied, then $\mu_G$ satisfies $\operatorname{WSM}(C)$ for some $C>0$. This can be proven via modifying the percolation argument in the proof of Proposition \ref{proposition672}. In contrast, condition $\operatorname{WSM}(C)$ can be weak and may not imply the presence of a strong external field as in Assumption \ref{primaryassumptionpaper}. 
\end{remark}

The spatial mixing property for a lattice spin system has many direct implications for mixing times and functional inequalities, see for example \cite{MartinelliOlivieri1994OnePhase,Martinelli1999Glauber}. But much less is known in the presence of random disorders.
What kind of graphs will satisfy $\operatorname{WSM}(C)$ for some $C$ at small disorder (smaller than the threshold for a worst case percolation argument in Section \ref{apercolationargument})? The current worked out examples are on integer lattices $\mathbb{Z}^d$ \cite{ding2021exponential,ding2024long}. Yet this question remains largely unexplored on a general graph that may have super-polynomial growth or strong local inhomogeneity.

We restrict ourselves to graphs of sub-exponential volume growth:

\begin{Definition}
Let $\alpha\in(0,1)$. We say a graph $G=(V,E)$ has $\alpha$-stretched-exponential growth if there exists $C_\alpha>0$ such that for each $v\in V$, $$|B_r(v)|\leq\exp(C_\alpha r^\alpha)\text{ for all }r\in\mathbb{N}_+.$$ 
\end{Definition}
In particular, a graph of polynomial growth has $\alpha$-stretched exponential growth for all $\alpha\in(0,1)$. For these graphs of subexponential growth, we prove:

\begin{theorem}\label{theorem231230}
Let $\alpha\in(0,1)$ and let $G=(V,E)$ be a graph with
$\alpha$-stretched-exponential growth. Assume that
$h=(h_u)_{u\in V}$ are i.i.d. random variables.

Suppose that we are given an ordering $v_1,\ldots,v_{|V|}$ of $V$,
chosen independently of the realization of $h$, such that, for
every $1\leq i\leq |V|$, the RFIM measure on the induced subgraph
\[
G^{(i)}:=G[\{v_1,\ldots,v_i\}]
\]
with free boundary conditions and external field
$(h_{v_j})_{1\leq j\leq i}$ satisfies $\operatorname{WSM}(C)$
for a common constant $C>0$ independent of $i$ and $|V|$.
Here the balls and their boundaries in the definition of
$\operatorname{WSM}(C)$ are taken with respect to $G^{(i)}$.

Then for any $\epsilon>0$ and $\delta\in(0,1)$, there exists
a randomized algorithm taking $G$, $\beta$, $h$, and the given
ordering as input, which runs in time $O(|V|^c)$ for
$c=c(\epsilon,\delta,\beta,C,\alpha,C_\alpha)>0$ and generates
a random configuration
$\sigma^{\operatorname{alg}}\in\{-1,+1\}^V$
with law $\mu_G^{\operatorname{alg}}$ such that, with
$\mathbb{P}_h$-probability at least $1-\delta$,
\[
d_{TV}(\mu_G,\mu_G^{\operatorname{alg}})\leq\epsilon.
\]
This algorithm adapts the sequential sampling construction of
\cite{el2026fast}, Theorem 1.4, to the given sequence of induced
subgraphs, with a different choice of parameters.
\end{theorem}

Theorem \ref{theorem231230} establishes polynomial-time sampling
for RFIM on graphs with $\alpha$-stretched-exponential growth,
provided that $\operatorname{WSM}(C)$ holds uniformly along a
given sequence of induced subgraphs. This extends the sampling
result of \cite{el2026fast}, Theorem 1.4, from lattice domains
to this setting.
\begin{remark}
The strong external field assumption in
Assumption \ref{primaryassumptionpaper} is replaced by uniform
$\operatorname{WSM}(C)$ along a given increasing sequence
of induced subgraphs.
The result provides a polynomial-time sampling algorithm
based on successive runs of Glauber dynamics on these subgraphs;
it does not assert polynomial-time mixing of Glauber dynamics
on $G$ from an arbitrary initial configuration.
The boundedness assumption \eqref{boundexternalfield} is not
required, so Gaussian external fields are also allowed whenever
the stated uniform WSM assumption holds.
\end{remark}

\begin{remark}
In the proof of the weak Poincaré inequality in
Theorem \ref{statementtheorem906}, the decay supplied by
$\operatorname{WSM}(C)$ must compensate for volume growth
in the combinatorial estimate of
Proposition \ref{themaincombinatorial}.
For graphs with $\alpha$-stretched-exponential growth,
$\alpha\in(0,1)$, this estimate works for every fixed $C>0$.
For graphs of exponential volume growth, the present approach
would require a restriction on $C$ relative to the growth rate,
as well as further modifications to the combinatorial argument.
\end{remark}

\subsection{Overview of the proof}

The RFIM measure $\mu_G$ is difficult to study by itself since it does not satisfy the Dobrushin condition or standard notions of spectral independence when $\beta$ is large. Motivated by the recent application of localization schemes of Chen and Eldan \cite{chen2022localization}, we will seek a continuous (and stochastic) deformation of $\mu_G$ to a target measure that is already known to satisfy a functional inequality such as Poincaré inequality, and that the deformation approximately preserves certain statistics, such as variance, of $\mu_G$. 

\subsubsection{Glauber dynamics at large disorder on general graphs}
How can we choose a proper localization scheme for RFIM? The first candidate choice is the stochastic localization scheme in \cite{chen2022localization} where we use a Brownian motion on the localizing path. The exact path of localization still requires a delicate choice: one standard deformation is to use the localization path to weaken the quadratic interaction, so that $\mu_G$ is deformed to 
$$
\nu^1_t(\sigma)\propto\exp\{\beta(1-t)\sum_{(u,v)\in E}\sigma_u\sigma_v+\sum_{u\in V}(h_u+y_u^1(t))\sigma_u\}
,\quad\forall \sigma\in\{-1,+1\}^V,$$ where $(y_u^1(t))_{u\in V}$ is a random external field defined by stochastic integrals of the adjacency matrix of $G$ with respect to a Brownian motion, see \cite{chen2022localization}, Fact 14 for the formula. This deformation, while useful for many important models (cf. \cite{eldan2022spectral}, \cite{chen2022localization}), does not work well for RFIM $\mu_G$ because the external field $h_u$ is changed to a random external field $h_u+y_{u}^1(t)$ that differs at each time $t$, and the quadratic interaction is also changed along the path. We find it hard to uniformly control $\nu_t^1$ over all realizations of the stochastic localization path $y_u^1(t)$ with respect to \textbf{a single quenched realization} of the external field $(h_u)_{u\in V}$. 
In contrast, in the RFIM setting the authors of \cite{el2026fast} chose another stochastic localization path that fixes the quadratic interaction and changes only the external field. Then $\mu_G$ deforms to
$$
\nu^2_t(\sigma)\propto\exp\{\beta\sum_{(u,v)\in E}\sigma_u\sigma_v+\sum_{u\in V}(h_u+y_u^2(t))\sigma_u\}
,\quad\forall\sigma\in\{-1,+1\}^V.$$ The distribution of this random external field $y_u^2(t)$ is stated in Lemma \ref{lemma856}.
In this second localization, although the external field gains more randomness, the quadratic form is unchanged and the property $\operatorname{WSM}(C)$ can be passed along the localization path in a high probability sense (but not for every realization of the localization path).
 Thanks to this high probability preservation of $\operatorname{WSM}(C)$, the authors of \cite{el2026fast} can transfer functional inequalities for $\nu_T^2$ for a large $T$ back to weak functional inequalities for $\mu_G$. This idea is called virtually increasing the external field since the measure $\nu_T^2$ is a RFIM with the same quadratic term but having a much stronger disorder in the random field. 
Via this strong disorder, they used a coarse graining idea on $\mathbb{Z}^d$ to verify that, for a sufficiently large $T$, the measure $\nu_T^2$ satisfies a Poincaré inequality. This transfers back into weak Poincaré inequalities for $\mu_G$ on $\mathbb{Z}^d$, which shows rapid mixing of Glauber dynamics from a warm start. Under uniform $\operatorname{WSM}(C)$ along a suitable increasing
sequence of domains, they then constructed a polynomial-time
sampling algorithm for $\mu_G$ by running Glauber dynamics
successively on these domains with free boundary conditions. However, the final step of coarse graining of $\nu_T^2$ significantly restricts the geometry of $G$ where the algorithm converges.

We take a different view towards designing localizing schemes of RFIM measures on general graphs. We need a localization path that (1) maintains a good control of the relevant covariance/correlation matrix of $\nu_t$ along the localization path, and the estimate holds for any possible pinning with respect to a single quenched external field $h$; and (2) the terminal measure is easy to analyze. The very recent paper \cite{chen2026edge} introduced an edge field localization scheme which is a piecewise deterministic jump Markov process that does not use SDEs. Notably, this localization path does not change the external field of RFIM and allows a quenched probability control over $h$. The construction is as follows. We first re-parametrize RFIM as a probability measure on $\{0,1\}^V$ rather than $\{-1,+1\}^V$. Then for any target measure $\mu$ on $\{0,1\}^V$, proceed as follows. Pick a family of events on the configuration $\sigma\in\{0,1\}^V$: $A_{(u,v)}:=\{\sigma_u=\sigma_v=1\}$ for each edge $(u,v)\in E$. Let $\mathcal{A}$ denote the collection of these events $A_{(u,v)}$. 
For the configuration $\sigma$ denote by 
$$
Z_{A_{(u,v)}}(\sigma)=\mathbf{1}_{\sigma\in A_{(u,v)}},\quad Z(\sigma)=\{A_{(u,v)}\in\mathcal{A}:\sigma\in A_{(u,v)}\}.
$$
We first sample $X\sim\mu$, then for each $A_{(u,v)}\in\mathcal{A}$
we independently sample a uniform random variable $U_{(u,v)}\sim\operatorname{Unif}[0,1]$. Then we define the set of events revealed at time $t$:
$$
T_t:=\{A_{(u,v)}\in\mathcal{A}
:X\in A,U_{A_{(u,v)}}\leq t\}.$$Then the localization path is the following posterior distribution, for each $S\subset\mathcal{A}$:
$$
\mu_t^S(\sigma)=\mathbb{P}(X=\sigma\mid T_t=S)\propto\mu(\sigma)\cdot (1-t)^{|Z(\sigma)|}\cdot\mathbf{1}[S\subset Z(\sigma)].
$$
Then observe that $\mu_t^S$ is still an Ising model on $G$ with vertices along edges that are associated to an event in $S$ pinned at 1, and the quadratic interactions reduce from $4\beta$ to $4\beta+\ln(1-t)$. It remains to upper bound the largest eigenvalue of a certain correlation matrix of $\mu_t^S$ uniformly over all pinnings on $S$ and uniformly over all quadratic interaction coefficients $4\beta+\ln(1-t)$ whenever it is within $[0,4\beta]$. As the random  external fields $h$ are unchanged along the path, we can apply a percolation argument and take a grand coupling of all these tilted RFIM models to deduce, with very high probability over $h$, a uniform upper bound for the largest eigenvalue of the correlation matrix. Finally, we take the terminal time $\theta_*=1-\exp(-4\beta)$, where the conditional measure is a product measure for any pinning $S$ and approximate tensorization of variance directly follows. This approach only relies on a percolation argument and completely bypasses fine local geometries of $\mathbb{Z}^d$. Unlike other terminal-time choices based on high-temperature tensorization, the current edge-field denoising flow is stopped when the posterior becomes a product measure; therefore, no separate high-temperature case results are needed.

\subsubsection{Weak spatial mixing on graphs of subexponential growth.}
When we only assume the RFIM disorder satisfies $\operatorname{WSM}(C)$, the localization scheme in the previous paragraph does not directly follow. Although we know that $\operatorname{WSM}(C)$ holds for one value $\beta$, it does not necessarily imply $\operatorname{WSM}(C)$ for smaller $\beta$ or that we can take a supremum over $\beta$ inside the expectation over $h$ when defining 
$\operatorname{WSM}(C)$.
The arbitrary pinning also may not preserve $\operatorname{WSM}(C)$. Therefore, we begin with the SDE-driven localization scheme $\nu_t^2$ as in \cite{el2026fast} to virtually increase the external field, and then for large $T$ we invoke Theorem \ref{refinementtheorem} to derive a Poincaré inequality for $\nu_T^2$ as it has large disorder. When the graph $G$ has $\alpha$-stretched-exponential growth
for some $\alpha\in(0,1)$, SDE localization path $\nu_t^2$ transfers $\operatorname{WSM}(C)$ estimates effectively until a large terminal time $T$ where Theorem \ref{refinementtheorem} can be used.

\subsection{Facts, notations and conventions}
\subsubsection{Graph distance and boundary} 

Let $G=(V,E)$ be a graph. For any two vertices $u,v\in V$ we use $d(u,v)$ to denote the graph distance between $u$ and $v$ in $G$.

For a subset $A\subset V$, denote by $\partial A:=\{v\in G:v\notin A,\exists w\in A,(v,w)\in E\}$ the boundary of $A$.
For any fixed configuration $\tau\in\{-1,1\}^{\partial A}$ on $\partial A$, we denote by $\mu_A^\tau$ the restriction of the Ising model $\mu_G$ to $A$ with the boundary condition $\tau$.

Let $(u,v)\in E$ be an edge of $G$, for any $l\in\mathbb{N}_+$ we denote by $B((u,v),l)$ the $l$-neighborhood of the edge $(u,v)$:

\begin{equation}\label{neighborhood}
B((u,v),l):=\{w\in G:d(w,u)\leq l\text{ or }d(w,v)\leq l\}=B(u,l)\cup B(v,l).
\end{equation}

\subsubsection{Markov chain mixing times}

Let $P$ be a Markov operator on $\Omega$ with invariant measure $\nu$, $P$ being reversible with respect to $\nu$. Let $\mu$ be an initial distribution that is absolutely continuous with respect to $\nu$.

The total-variation mixing time is defined by 
$$
t_{mix}(P,\epsilon;\mu)=\min\{t>0;|P^t[\mu](A)-\nu(A)|\leq\epsilon,\forall A\subset\Omega\},
$$
and we set 
$$
t_{mix}(P,\epsilon)=\max_{x\in\Omega}t_{mix}(P,\epsilon,\delta_x).
$$
Then we have the following standard fact, see, for example, \cite[Chapters 12--13]{LevinPeresWilmer2017}: 

\begin{fact}\label{firststandardfact}
    Assume that for all $x\in\Omega$, $\nu(\{x\})\geq\eta$. Assume additionally that $P$ has nonnegative spectrum. Then 
    $$
t_{mix}(P,\epsilon)\leq C\mathbf{gap}(P)^{-1}(\log(1/\eta)+\log(1/\epsilon)),
    $$and 
    $$
t_{mix}(P,\epsilon)\leq C\rho_{LS}(P)^{-1}(\log\log(1/\eta)+\log(1/\epsilon))
    .$$
\end{fact}

\subsubsection{FKG and coupling.}\label{howtousecouplinghere} 
If we assume that all the RFIM interactions are positive, then $\mu_G$ satisfies useful correlation inequalities including the FKG inequality. That is, for a subset $A\subset V$ and a non-decreasing function $\varphi:\{-1,1\}^A\to\mathbb{R}$, the mean of $\varphi$ is monotone in the boundary condition: $\mu_A^\tau(\varphi)\leq\mu_A^{\tau'}(\varphi)$ whenever boundary conditions satisfy $\tau\leq\tau'$ pointwise. 

Again assuming monotonicity, for the two boundary conditions $\tau\leq \tau'$, it is standard to use the monotone coupling to couple $\mu_A^\tau$ and $\mu_A^{\tau'}$. Specifically, we fix an ordering $v_1,\cdots,v_k$ of vertices in $A$ and we then draw uniform random variables $U_1,\cdots,U_k$ independently. Then the coupling $(\sigma^\tau,\sigma^{\tau'})$ can be constructed sequentially where for each $i\geq 1$ we shall sample $\sigma_{v_i}^\tau$, $\sigma_{v_i}^{\tau'}$ from the conditional distributions given $(\sigma_{v_j}^\tau)_{j<i}$, $\tau$ and $(\sigma_{v_j}^{\tau'})_{j<i},\tau'$ where we use the common randomness in $U_i$. The coupling is clearly monotone in the boundary conditions: $\sigma^\tau\leq\sigma^{\tau'}$ whenever $\tau\leq\tau'$. 

Without assuming monotonicity, let $A\subset V$ and consider a class of Ising measures $\mu_1,\mu_2,\cdots,\mu_t$ on $A$ with boundary conditions $\tau_1,\cdots,\tau_t$ on $\partial A$, where each Ising measure may have different interactions and external fields, and each $\mu_i$ may be further pinned at some $A_i\subset A$ by some value $\tau_i'\in\{-1,1\}^{A_i}$. Then we can use the same uniform random variable to couple $\mu_1,\cdots,\mu_t$ as above. Specifically, we again fix an ordering $v_1,\cdots,v_k$ of $A$ and draw uniform random variables $U_1,\cdots,U_k$ independently. Then we construct a coupling of $(\sigma^1,\cdots,\sigma^t)$ where $\sigma^1\sim \mu_1,\cdots,\sigma^t\sim\mu_t$ sequentially so that for each $i\geq 1$ we sample $\sigma^1_{v_i},\cdots,\sigma^t_{v_i}$ from their respective conditional distributions given $(\sigma^r_{v_j})_{j<i},1\leq r\leq t$ and $\tau_1,\cdots,\tau_t$, using $U_i$ as common randomness. If $v_i$ is already pinned in $\mu_r$ then we output the pinned value for $\sigma^r_{v_i}$.

\subsection{Change of coordinates for RFIM}

The current RFIM measure is defined on $\{-1,+1\}^V$ but we will frequently change its coordinates to a measure on $\{0,1\}^V$. The following fact summarizes the change of coordinates:

\begin{fact}\label{fact2690}
    Let $\mu_G$ be the RFIM measure \eqref{firstrfimmeasure} defined on $\{-1,+1\}^V$. Then it admits an equivalent expression
$$
\mu_G(\sigma')\propto \exp\{
4\beta\sum_{(u,v)\in E}\sigma_u'\sigma_v'+\sum_u (2h_u-2\beta d_u)\sigma_u'
\},\quad \forall \sigma'\in\{0,1\}^V,
$$    where $\sigma_u'=\frac{\sigma_u+1}{2}$ for each $u\in V$. Here $d_u$ is the degree of $u$ in $G$ for each $u\in V$.
Similarly, the RFIM measure $\mu_{\tilde{G}}$ defined in \eqref{RFIMFORMIXED} admits an equivalent expression
$$
\mu_{\widetilde{G}}(\sigma')\propto \exp\{
4\sum_{(u,v)\in E}\beta_{uv}\sigma_u'\sigma_v'+\sum_u (2h_u-2\sum_{v:(u,v)\in E}\beta_{uv})\sigma_u'
\},\quad \forall \sigma'\in\{0,1\}^V.
$$
\end{fact}
Namely, after the change of coordinates, the external fields are shifted by a deterministic, vertex dependent constant but remain independent.

\section{Edge-field localization schemes and spectral independence}\label{sections344}

In this section we introduce the general framework of edge tilted localization process of \cite{chen2026edge}. Then we introduce an edge field localization process for the RFIM $\mu_G$ and provide criteria for its spectral stability.

\subsection{Entropy conservation and localization schemes}

Consider a domain $D\subseteq\mathbb{R}$ and a convex function $\phi:D\to\mathbb{R}$. Consider $\mu$ a distribution on a finite set $\Omega$ and sample $X\sim\mu$. For any real-valued function $f:\Omega\to\mathbb{R}$, define the $\phi$-entropy of $f$ with respect to $\mu$ by 
$$
\operatorname{Ent}_\mu^\phi[f]:=\operatorname{Ent}^\phi[f(X)]:=\mathbb{E}[\phi(f(X))]-\phi(\mathbb{E}[f(X)]).
$$
If $\nu$ is a  probability measure on $\Omega$ absolutely continuous with respect to $\mu$, we define the $\phi$-divergence between $\nu$ and $\mu$ via 
$$
D_\phi(\nu\mid\mid\mu):=\operatorname{Ent}_\mu^\phi\left[\frac{\nu}{\mu}\right].
$$

In this paper we focus on the following two specific cases
\begin{itemize}\item $\phi(x)=\chi^2(x):=x^2$, so that $\chi^2$-entropy of $f$ is equivalent to the variance of $f$, and 
\item $\phi(x)=\mathbf{KL}(x):=x\log x$, which induces the KL-divergence
$$
D_{KL}(\nu\mid\mid\mu)=\sum_{\sigma\in\Omega}\nu(\sigma)\log\frac{\nu(\sigma)}{\mu(\sigma)}.
$$

\end{itemize}

We now introduce some notations from \cite{chen2026edge} on the conservation and decay of entropy for a Markov operator. To be consistent with the notations in \cite{chen2026edge}, we use the state space $\{0,1\}^n$ rather than $\{-1,1\}^n$. The formula for the RFIM measure $\mu_G$ on the new coordinate $\{0,1\}^n$ is explicitly written in Fact \ref{fact2690}. 

\begin{Definition}(Approximate tensorization of $\phi$-entropy) We say that a distribution $\mu$ on $\{0,1\}^n$ satisfies $K$-approximate tensorization of $\phi$-entropy if for any function $f:\Omega\to\mathbb{R}$ and any $X\sim\mu$,
$$
\operatorname{Ent}^\phi[f(X)]\leq K\cdot\sum_{i=1}^n \mathbb{E}[\operatorname{Ent}^\phi[f(X)\mid X_{-i}]],
$$where we denote by $X_{-i}$ the configuration $X$ restricted to $[n]\setminus\{i\}$.
    
\end{Definition}

 \begin{Definition}(Decay of entropy) Consider a Markov chain $(X_t)_{t\geq 0}$ on $\Omega$ having transition matrix $P$ and stationary distribution $\mu$. Then we say that it satisfies decay of $\phi$-entropy  with rate $\kappa$ if for all functions $f:\Omega\to\mathbb{R}$,
 $$
\operatorname{Ent}_\mu^\phi[Pf]\leq (1-\kappa)\operatorname{Ent}_\mu^\phi[f].
 $$\end{Definition}

Then we introduce the general notion of a localization process in \cite{chen2026edge}. We will later specify the exact localization scheme that we shall use.
\begin{Definition}
Consider $\mu$ a probability distribution on $\Omega$. Then a localization scheme having $\mu$ as target distribution is made up of a pair of continuous time processes:
\begin{itemize}
    \item The noising process $(X_t)_{t\in[0,1]}$. This is a Markov process such that $X_0\sim\mu$ and $X_1$ follows a Dirac measure.

    \item A denoising process $(Y_t)_{t\in[0,1]}$. This process is defined as the reversal of $(X_t)_{t\in[0,1]}$: $Y_\theta=X_{1-\theta}$. Thus $Y_0$ is a Dirac measure and $Y_1\sim\mu$. 
\end{itemize}
Since $X_t$ is a time reversal of $Y_t$, we use the denoising process $(Y_t)_{t\in[0,1]}$ to represent the localization scheme and call it the localization process.
\end{Definition}

\begin{Definition}(Approximate conservation of $\phi$-entropy)
Let $(Y_t)_{t\in[0,1]}$ be a denoising process with target distribution
$\mu$ on $\Omega$, and let $\theta\in[0,1]$.
The process satisfies $R$-approximate conservation of $\phi$-entropy
up to time $\theta$ if, for every $f:\Omega\to D$,
$$
\operatorname{Ent}_\mu^\phi[f]
\leq R\,\mathbb{E}\left[\operatorname{Ent}^\phi[f(Y_1)\mid Y_\theta]\right].
$$
\end{Definition}

We would like the localization path $Y_t$ to have the following nice property:

This general notion of localization schemes gives rise to a boosting of mixing properties from a parameter regime where mixing is known, to a new parameter regime where we expect to prove mixing.
\begin{lemma}\label{lemma3500}(\cite{chen2025rapid1}, Lemma 3.13) Consider $(Y_t)_{t\in[0,1]}$ a denoising process for a distribution $\mu$. Assume we can find some $\theta\in[0,1]$ such that
\begin{enumerate}
    \item The conditional distribution $\operatorname{Law}(Y_1\mid Y_\theta)$ satisfies $K$-approximate tensorization of $\phi$-entropy.

    \item The denoising process $(Y_t)_{t\in[0,1]}$ satisfies $R$-approximate conservation of $\phi$-entropy up to time $\theta$.
\end{enumerate}

Then $\mu$ satisfies (KR)-approximate tensorization of $\phi$-entropy. 
    
\end{lemma}

The following property is very useful in verifying approximation conservation of $\phi$-entropy:

\begin{Definition}
    Consider the denoising process $(Y_t)_{t\in[0,1]}$ for a measure $\mu$ on $\Omega$. Then the process $(Y_t)_{t\in[0,1]}$ is said to be $\phi$-entropically stable with rate $C$ at $\theta\in[0,1]$ if for any given function $f:\Omega\to\mathbb{R}$, denoting by $F=f(Y_1)$, then for any configuration $T$ with $\mathbb{P}(Y_\theta=T)>0$, we have
    $$
\lim_{h\to 0^+}\frac{1}{h}\operatorname{Ent}^\phi[\mathbb{E}[F\mid Y_{\theta+h}]\mid Y_\theta=T]
\leq\frac{C}{1-\theta}\operatorname{Ent}^\phi[F\mid Y_\theta=T].
    $$
\end{Definition}
The property of $\phi$-entropic stability implies approximate conservation of $\phi$-entropy, as outlined in the following lemma:

\begin{lemma}\label{lemma3760}(\cite{chen2025rapid1}, Theorem 3.16) Given $C:[0,1]\to\mathbb{R}_{\geq 0}$. Assume that a denoising process $(Y_t)_{t\in[0,1]}$ is $\phi$-entropically stable with rate $C(t)$ throughout $t\in[0,1]$. Then for each given $\theta\in[0,1]$, the process $(Y_t)_{t\in[0,1]}$ will satisfy $R$-approximate conservation of $\phi$-entropy up to time $\theta$ where we take 
$$
R=\exp\left(\int_0^\theta \frac{C(t)}{1-t}dt\right).
$$
    
\end{lemma}

Therefore, in order to deduce approximate tensorization of $\phi$-entropy of $\mu$, we only need to construct a denoising process $Y_t$ which is $\phi$-entropically stable.

\subsubsection{From approximate tensorization to functional inequalities for Glauber dynamics}

We recall here the fact that approximate tensorization of variance ($\chi^2$
 entropy) implies a Poincaré inequality for Glauber dynamics:

 \begin{fact}\label{fact3900}
A distribution $\mu$ on $\{0,1\}^n$ satisfies $C_1$- approximate tensorization of variance if and only if the spectral gap for the Glauber dynamics for $\mu$ satisfies $\mathbf{gap}_{GL}\geq\frac{1}{C_1n}.$      
 \end{fact}

 A similar statement holds for modified log-Sobolev inequalities:
 \begin{fact}\label{fact4160}(\cite{chen2021optimal}, Fact 3.5)
If a distribution $\mu$ on $\{0,1\}^n$ satisfies $C_1$- approximate tensorization of entropy, then the modified log-Sobolev inequality holds for the Glauber dynamics with constant $\rho_0\geq\frac{1}{C_1n}.$      
 \end{fact}

\subsection{A special case: vertex field denoising}

An important case of localization schemes is the vertex-tilting field dynamics introduced in \cite{chen2024rapid}. Our edge-tilting localization scheme will also be built on this vertex field version, so we give a quick review here. We consider distributions on $\{0,1\}^n$ and identify a configuration $\sigma\in\{0,1\}^n$ with the subset $\{v\in [n]:\sigma_v=1\}$. 

\begin{Definition}
    Consider a distribution $\mu$ on $\{0,1\}^n$ and a sample $X\sim\mu$. Assign, for each $v\in[n]$, an independent random variable $U_v\sim\operatorname{Uniform}[0,1]$. Then the vertex field denoising process with respect to $\mu$ is defined by the following process $(Y_t)_{t\in[0,1]}$:
    $$
Y_t:=\{v\in[n]:X_v=1\wedge U_v\leq t\}.
    $$
\end{Definition}

For a $\theta\in[0,1]$, the tilted distribution $\theta*\mu$ is defined as
$$
(\theta*\mu)(\sigma)\propto \mu(\sigma)\cdot\theta^{|\sigma|},\quad\forall \sigma\in\{0,1\}^n,
$$where $|\sigma|:=|\{v\in[n]:\sigma_v=1\}|.$

\begin{lemma}(\cite{chen2026edge}, Proposition 2.14) For a distribution $\mu$, let $(Y_t)_{t\in[0,1]}$ be the corresponding vertex field denoising process. Then  we have
$$
\operatorname{Law}(Y_1\mid Y_t)=(1-t)*\mu^{Y_t},
$$
 such that $\mu^{Y_t}$ is the conditional law obtained by setting all variables in $Y_t$ pinned to be 1.   
\end{lemma}

Spectral stability of vertex-field denoising process has the following characterization:

\begin{lemma}\label{lemma88}(\cite{chen2025rapid2}, Proposition 3.3)
    Let $(Y_t)_{t\in[0,1]}$ be a vertex-field denoising process with respect to $\mu$ on $\{0,1\}^n$. Then for any fixed $\theta\in[0,1]$, the following two claims are equivalent:
    \begin{itemize}
        \item $(Y_t)_{t\in[0,1]}$ is spectrally stable with rate $C$ at time $\theta$;
        \item Given any subset $S$ with $\mathbb{P}[Y_\theta=S]>0$,
        $$
\operatorname{Cov}((1-\theta)*\mu^S)\preceq C\cdot\operatorname{diag}\{\mathbf{m}((1-\theta)*\mu^S)\},
        $$ where for a distribution $\nu$, $\mathbf{m}(\nu):=\mathbb{E}_{X\sim\nu}[X]$ is its mean vector. 
    \end{itemize}
\end{lemma}

\subsection{Edge-field tilting and its localization schemes}

The vertex-field denoising process introduced in the last section changes the strength of the external field of an Ising model but keeps quadratic interaction unchanged. This is not sufficient for the RFIM where we wish to weaken the quadratic interaction but to keep the external field fixed. The recent paper 
\cite{chen2026edge} pointed out that our goal can be fulfilled by considering a tilting of the edge-field instead of vertex-field, so we remove edges rather than vertices. Before presenting our concrete constructions of edge-field denoising, we outline the abstract framework of \cite{chen2026edge}.

For a given distribution $\mu$ on $\{0,1\}^V$ where $V$ is a ground set, let $\mathcal{A}$ be a collection of events, so that each $A\in\mathcal{A}$ is a subset of $\{0,1\}^V$. We take $X\sim\mu$ and for any event $A\in\mathcal{A}$ define $Z_A=Z_A(X):=\mathbf{1}[X\in A]$. Let $Z=Z(X):=(Z_A)_{A\in\mathcal{A}}$.

We let $\pi$ denote the joint distribution of $(X,Z)$, so that 
$$
\pi:=\operatorname{Law}(X,Z(X)),
$$so that $\pi$ is a distribution over $\{0,1\}^{V\cup\mathcal{A}}$. Then by construction, $\mu=\pi_V$ and $\operatorname{Law}(Z)=\pi_\mathcal{A}$.

\begin{Definition}(Event-field dynamics, \cite{chen2026edge}) Let $\theta\in[0,1]$. The event-field dynamics for $\mu$ with respect to $\mathcal{A}$ and tilt $\theta$ is the following defined Markov chain on the state space $\Omega(\mu)$:

From the current state $X_t\in\Omega(\mu)$, the next state $X_{t+1}$ is generated via
\begin{enumerate}
    \item ($V\to\mathcal{A}$). Set $Z_t:=Z(X_t)=\{A\in\mathcal{A}\mid X_t\in A\}.$
    \item (Down-$\mathcal{A}$) Then generate a random subset $T\subset Z_t$ where we independently remove each $A\in Z_t$ with probability $\theta$;
    \item (Up-$\mathcal{A}$) Sample $Z_{t+1}\sim(\theta*\pi_\mathcal{A})(\cdot\mid Z_{t+1}\supseteq T)$, where we define 
    $$
\theta*\pi_\mathcal{A}(Z)\propto \pi_A(Z)\cdot \prod_{A\in\mathcal{A}}\theta^{\mathbf{1}(Z_A=1)},\quad \forall Z\in\{0,1\}^{\mathcal{A}},$$
     the distribution obtained from $\pi_\mathcal{A}$ by tilting the occurrence of each event $A$ by   $\theta$. 
    \item Sample $X_{t+1}\sim\pi_V(\cdot\mid Z_{t+1})$.
\end{enumerate}
    
\end{Definition}

The only example that we will use in this paper is the following edge-field dynamics:
\begin{Example}\label{exam454}(Edge-field dynamics) Let $G=(V,E)$ be a graph and $\mu$ a distribution on $\{0,1\}^V$. We set
$$
\mathcal{A}:=\{A_{uv}\mid (u,v)\in E\}, \quad\text{where } A_{uv}:=\{\sigma:\sigma_u=\sigma_v=1\},\quad\forall (u,v)\in E.
$$ Given $T\subseteq \mathcal{A}$ generated in the (Down-$\mathcal{A}$) step, we can compute that 
$$\begin{aligned}
\mathbb{P}(X_{t+1}=\sigma\mid T)&\propto \mu(\sigma)\cdot\theta^{|\mathcal{A}(\sigma)|}\cdot\mathbf{1}[\mathcal{A}(\sigma)\supseteq T],
\end{aligned}$$
    where $\mathcal{A}(\sigma):=\{A_{uv}\in\mathcal{A}:\sigma\in A_{uv}\}.$
\end{Example}

Next we recall the definition of an event-field denoising process from \cite{chen2026edge}. This is mostly a projection of the vertex-field denoising process on the collection of events $\mathcal{A}$.

\begin{Definition}\label{definition45566}(Event-field denoising) Let $\mu$ be a distribution on $\{0,1\}^V$ and consider a collection of events $\mathcal{A}$. Denote by $\pi$ the joint distribution on $\{0,1\}^{V\cup\mathcal{A}}$ and let $(Y_t')_{t\in[0,1]}$ be the vertex-field denoising process for the marginal $\pi_{\mathcal{A}}$ with $\pi_{\mathcal{A}}=\operatorname{Law}(Y_1')$.
    Then event-field denoising process for $\mu$ through $\mathcal{A}$, which we denote by $(Y_t)_{t\in[0,1]}$ is defined via
\begin{enumerate}
    \item Sample $Y_1\sim\pi_V(\cdot\mid Y_1')$,
    \item For $t\in[0,1)$, we set $Y_t=Y_t'$. 
\end{enumerate} In other words, the process follows the vertex field denoising process applied to the event indicators, and finally at $t=1$ projects onto the vertices in $V$ via the conditional law $\pi_V(\cdot\mid Y_1')$.
\end{Definition}
Since $Y_t=Y_t'$ for all $t\in[0,1)$, the spectral stability of $(Y_t')_{t\in[0,1]}$ implies the spectral stability of $(Y_t)_{t\in[0,1]}$:
\begin{lemma}\label{lemma2.16}($\phi$-Entropic stability upgrades, \cite{chen2026edge}, Proposition 4.12) 
    Let $(Y_t)_{t\in[0,1]}$ and $(Y_t')_{t\in[0,1]}$ be the two denoising processes in Definition \ref{definition45566}. Then for any $\theta\in[0,1)$, suppose that $(Y_t')_{t\in[0,1]}$ is $\phi$-entropically stable with rate $C$ at time $\theta$, then $(Y_t)_{t\in[0,1]}$ is also $\phi$-entropically stable with rate $C$ at time $\theta$.
\end{lemma}

Then we specialize to the edge-field dynamics in Example \ref{exam454} and define its own denoising process:

\begin{Definition}\label{definition569}(Edge-field denoising process) Let $G=(V,E)$ be a graph and $\mu$ a distribution on $\{0,1\}^V$. Recall that we define 
$$
\mathcal{A}:=\{A_{uv}\mid (u,v)\in E\},\quad\text{where }A_{uv}:=\{\sigma:\sigma_u=\sigma_v=1\},\quad \forall (u,v)\in E.
$$ Let $(Y_t)_{t\in[0,1]}$ be the event-field denoising process for $\mu$ with respect to this family $\mathcal{A}$. Then the resulting process $(Y_t)_{t\in[0,1]}$ is called the edge-field denoising process for $\mu$ on $G$.
\end{Definition}

In the following, we give an explicit description of the process $Y_t$. We shall use a notation $\theta\otimes \mu$ for the distribution obtained by tilting the interactions by a factor $\theta>0$:
$$\theta\otimes \mu(\sigma)\propto \mu(\sigma)\cdot \theta^{m(\sigma)},\quad\forall \sigma\in\{0,1\}^V,
$$where we denote by $m(\sigma):=|\{(u,v)\in E:\sigma_u=\sigma_v=1\}|$ the number of edges whose two endpoints are both equal to 1.

The edge field denoising process $(Y_t)_{t\in[0,1]}$ satisfies that its posterior distribution is the target distribution $\mu$ with tilted interaction:

\begin{Proposition}\label{proposition511}
    Let $(Y_t)_{t\in[0,1]}$ denote the edge-field denoising process. Then its posterior distribution satisfies 
    $$
\operatorname{Law}(Y_1\mid Y_t)=(1-t)\otimes \mu^{Y_t},
    $$where we recall that $\mu^{Y_t}$ denotes the conditional law of $\mu$ conditioning on the event that for every edge $(u,v)\in E$ with $Y_t(u,v)=1$, then $A_{uv}$ occurs, so that $\sigma_u=\sigma_v=1$.
\end{Proposition}

\begin{proof}
    For any $\sigma\in\{0,1\}^V$, denote by $Z(\sigma)\in\{0,1\}^\mathcal{A}$ the vector that indicates whether each $A_{uv}\in\mathcal{A}$ occurs in the configuration $\sigma$. Take $X\sim\mu$, recall that $\pi:=\operatorname{Law}(X,Z(X))$ on $\{0,1\}^{V\cup \mathcal{A}}$.

Let $(Y_t')$ be the vertex-field denoising process for $\pi_\mathcal{A}$. Then for any configuration $T$,
$$\begin{aligned}
\quad &\mathbb{P}[Y_1=\sigma\mid Y_t=T]\\&=\mathbb{P}[Y_1=\sigma\mid Y_1'=Z(\sigma)]\cdot\mathbb{P}[Y_1'=Z(\sigma)\mid Y_t'=T]\\&\propto \mathbb{P}[Y_1=\sigma\mid Y_1'=Z(\sigma)]\cdot (1-t)^{|Z(\sigma)|}\cdot \pi_\mathcal{A}(Z(\sigma))\cdot\mathbf{1}[T\subseteq Z(\sigma)]\\&\propto (1-t)^{m(\sigma)}\cdot \pi(\sigma, Z(\sigma))\cdot\mathbf{1}[T\subseteq Z(\sigma)]. 
\end{aligned}$$
    This implies $\operatorname{Law}(Y_1\mid Y_t)=(1-t)\otimes \mu^{Y_t}$ since $\sigma$ determines $Z(\sigma)$.
\end{proof}

\subsection{Spectral stability of edge-field dynamics}
Thanks to Lemma \ref{lemma2.16},
 we only need to prove the spectral stability of the denoising process $Y_t'$. Motivated by the criterion in Lemma \ref{lemma88}, we define the following correlation matrix for $Y_t'$:
\begin{Definition}\label{definition532}(Second-order correlation matrix) Let $\mu$ be a distribution on $\{0,1\}^V$ and $G=(V,E)$ be a graph. Then for $u,v\in V$, denote by $uv$ the event that 
$$
uv:=\{\sigma\in\{0,1\}^V\mid \sigma_u=\sigma_v=1\}.
$$ Then we define the second order correlation matrix $\operatorname{Cor}_\mu^{(2)}\in\mathbb{R}^{E\times E}$ via
$$
\forall (u,v),(w,z)\in E,\quad \operatorname{Cor}_\mu^{(2)}(uv,wz):=\begin{cases}
    \mu(wz\mid uv)-\mu(wz),&\text{ if }\mu(uv)>0,\\0,&\text{otherwise}.
\end{cases}
$$
    
\end{Definition}

\begin{remark}
    We define the correlation matrix $\operatorname{Cor}_\mu^{(2)}$ in this form due to the following fact: let $\mu$ be a distribution on $\{0,1\}^V$, then $(\operatorname{diag}(\mathbb{E}(\mu)))^{-1}\cdot \operatorname{Cov}_\mu=\operatorname{Cor}_\mu$.

    Although $\operatorname{Cor}_\mu^{(2)}$ is not a symmetric matrix, its eigenvalues are all real.
\end{remark}

Then we can translate spectral stability into the following criterion: 

\begin{corollary}\label{followingcriterion}(Spectral stability via second-order correlations) Let $(Y_t)_{t\in[0,1]}$ denote the edge-field denoising process for $\mu$ on the graph $G=(V,E)$. Then for fixed $\theta\in(0,1)$, assume that for every $\Lambda\subset V$ and any feasible pinning $\tau\in\Omega(\mu_\Lambda)$ we have that 
$$
\lambda_{max}(\operatorname{Cor}^{(2)}_{(1-\theta)\otimes \mu^\tau})\leq C,
$$
    then the process $(Y_t)_{t\in[0,1]}$ is spectrally stable with rate $C$ at time $\theta$.
\end{corollary}

\begin{proof}
This follows from combining Definition \ref{definition45566}, Lemma \ref{lemma2.16} and Lemma \ref{lemma88}. Note that a conditioning on the event field events $A_{uv}$ is equivalent to conditioning on the vertex events $\sigma_u=\sigma_v=1$.
\end{proof}

\subsection{Approximate tensorization at the end of the flow}

We first record how the parameter of RFIM model changes under the edge-field denoising process:

\begin{fact}
    Let $\mu$ be the RFIM measure $\mu_G$. Then for any $\theta\in(0,1)$ and any pinning $\tau$, the Ising measure $(1-\theta)\otimes \mu^\tau$ has quadratic coefficient $4\beta+\ln(1-\theta)$ (on the coordinate $\{0,1\}^V$ in its Hamiltonian form) on edges of $G$ whose both endpoints are not pinned by $\tau$. In particular, when we take $\theta_*=1-\exp(-4\beta)$, then $(1-\theta_*)\otimes\mu^\tau$ is a product measure. 
\end{fact}

\begin{proof}
By Fact \ref{fact2690}, in the coordinate $\{0,1\}^V$ the quadratic coefficient of $\mu=\mu_G$ in the Hamiltonian is $4\beta$. For each unpinned edge, the denoising process $Y_t$ shrinks the probability by multiplying by $1-\theta$ before normalizing, which means changing $4\beta$ to $4\beta+\ln(1-\theta)$ in the Hamiltonian coefficient.     
\end{proof}

For sake of completeness, we note here that when $\beta$ is sufficiently small, this model is within the realm of spectral independence where approximate tensorization of variance directly follows. We, however, do not use this estimate as we stop the denoising process at $\theta_*$ where the conditional law is the product measure. 
\begin{lemma}\label{lemma5660} Let $G=(V,E)$ be a graph of maximal degree at most $\Delta$, and for each $(u,v)\in E$ we are given constants $\beta_{uv}$ where $|\beta_{uv}|\leq\beta$ for all edges. Fix an arbitrary external field $h=(h_u)_{u\in V}$, an arbitrary subset $A\subset V$ and configuration $\tau\in\{0,1\}^A$. Then whenever $\beta\leq\frac{1}{\Delta}$, the measure $\mu_{\beta,h}$ defined on $\{0,1\}^{V\setminus A}$ via
$$
\mu_{\beta,h}(\sigma)\propto\exp(\sum_{(u,v)\in E}\beta_{uv}\sigma_u\sigma_v+\sum_{u\in V\setminus A}h_u\sigma_u),\quad\forall\sigma\in\{0,1\}^V:\sigma|_A=\tau
$$satisfies approximate tensorization of variance and approximate tensorization of entropy with constant 2.
\end{lemma}

\begin{proof}
    Let $Q$ denote the quadratic interaction matrix in $\mu_{\beta,h}$ where $Q_{uv}=\beta_{uv}$ if $(u,v)\in E$ and $u\notin A$ and $v\notin A$, and set $Q_{uv}=0$ otherwise. Then $\lambda_{max}(Q)\leq \beta\Delta$ and $\lambda_{min}(Q)\geq -\beta\Delta$, where $\lambda_{min}(Q),\lambda_{max}(Q)$ are the smallest and largest eigenvalues of $Q$. Then via Fact \ref{fact2690}, we can rewrite $\mu_{\beta,h}$ as an Ising model on $\{-1,1\}^V\setminus A$ where the quadratic interaction matrix is $\frac{1}{4}Q$. Then by \cite{anari2021entropic}, Theorem 12, $\mu_{\beta,h}$ satisfies approximate tensorization of variance and entropy with constant $\frac{1}{1-(\lambda_{max}(\frac{1}{4}Q)-\lambda_{min}(\frac{1}{4}Q))}\leq 2$.  
\end{proof}

Now we summarize what we have gotten so far, and what remains to be proven in the remaining sections. We will apply the edge field denoising process $Y_t$ to the RFIM $\mu_G$, where we take \begin{equation}\theta_*=1-\exp(-4\beta)\end{equation} as a threshold value. This value of $\theta_*$ is fixed throughout the paper. We need to upper bound the integral in $R$ in Lemma \ref{lemma3760} up to time $\theta_*$, and 
then tensorization of variance/entropy is free for $\operatorname{Law}(Y_1\mid Y_{\theta_*})$ since the latter is a product measure.
For the integral in $R$, by Proposition \ref{proposition511}, it suffices to control the second correlation matrix for $(1-\theta)\otimes\mu^\tau$ for all $0\leq\theta\leq\theta_*$ and all pinning $\tau$. In the next section, we will focus on control of the largest eigenvalue of that matrix.

\section{The largest eigenvalue of the correlation matrix with a random field}\label{apercolationargument}
In this section, we make the following assumption on the external field of RFIM, which is more general than Assumption \ref{primaryassumptionpaper} as we only require $|h_v|$ to be independent, but the signs of $h_v$ are not necessarily independent. This generalization is useful in Theorem \ref{statementtheorem906}. 
\begin{Assumption}\label{thiasssumptionnew} Let  $p_0\in(0,1)$ be given such that $p_0(\Delta-1)<1$, let $K=K(p_0)>0$ be a constant such that
\begin{equation}
    \rho=\rho(\beta,\Delta,K):=\frac{e^{\Delta\beta-K}}{e^{\Delta\beta-K}+e^{-\Delta\beta+K}},
\end{equation} is less than $\frac{p_0}{4}$. Suppose that the RFIM $\mu_G$ with external field $h$
is such that $(|h_u|)_{u\in V}$ are independent,
not necessarily identically distributed, and satisfy
\[
\sup_{u\in V}\mathbb{P}_h(|h_u|\leq K)<\frac{p_0}{2}.
\]
\end{Assumption}

\begin{Proposition}\label{proposition672}
    Let $(u,v)\in E$ be any edge of the graph $G=(V,E)$ of maximal degree $\Delta$. Assume that Assumption \ref{thiasssumptionnew} holds for the given $p_0\in(0,1)$, then we have, for any $m\geq 1$,
    \begin{equation}\begin{aligned}\label{lines6750}
\mathbb{P}&\left(\sup_{\theta\in[0,\theta_*],\Lambda,\tau_0\in\Omega(\Lambda)}
\sum_{(wz)\in E}|\operatorname{Cor}_{(1-\theta)\otimes\mu^{\tau_0}}^{(2)}(uv,wz)|\geq \Delta m\right)\\&\leq \frac{2}{(1-e^{-\alpha_*})(1-e^{-2\alpha_*})}e^{2\gamma_*}e^{-\alpha_*m},\end{aligned}
    \end{equation}
    where $\gamma_*:=\ln\frac{1-p_0}{p_0(\Delta-2)}$ and $\alpha_*=\frac{1}{2}\log\frac{(\Delta-2)^{\Delta-2}}{p_0(\Delta-1)^{\Delta-1}(1-p_0)^{\Delta-2}}.$
\end{Proposition}

\begin{proof}
For a fixed edge $(u,v)$ and
$\nu:=(1-\theta)\otimes\mu^{\tau_0}$, we construct a coupling
of $\sigma^1\sim\nu$ and $\sigma^2\sim\nu(\cdot\mid uv)$.
Keep the common pinnings fixed, set $\sigma_u^2=\sigma_v^2=1$,
and sample $\sigma_u^1,\sigma_v^1$ successively from their
marginal and conditional laws under $\nu$. All sampling uses
i.i.d. variables $(U_x)_{x\in V}$, uniform on $[0,1]$ and
independent of $h$.

Let $R$ be the set of revealed vertices, initially consisting of
the pinned vertices and $\{u,v\}$, and let
$D:=\{x\in R:\sigma_x^1\neq\sigma_x^2\}$.
Fix a deterministic priority rule on the vertices.
As long as an unrevealed vertex is adjacent to $D$, choose the
first such vertex $w$ according to this rule. Reveal its spin
in both configurations using the same, as yet unused, $U_w$
and the respective conditional Bernoulli laws given the
already revealed spins in that configuration. Add $w$ to $R$,
and also to $D$ if the two revealed spins differ.

When no such vertex remains, put $W:=V\setminus R$.
No vertex of $W$ is adjacent to $D$, so the revealed
configurations agree on $\partial W$.
By the spatial Markov property, the two conditional Gibbs
laws on $W$ are therefore identical: they have the same
interactions and fields, and the same boundary spins.
Sample this common law once and use the resulting configuration
for both copies on $W$. This final identity coupling may be
implemented by revealing the remaining vertices in the fixed
priority order, using the unused $U_x$.
The recursive construction has the required marginals by
induction on the number of unrevealed vertices.

If $|h_w|>K$, the single-site bound in
Assumption \ref{thiasssumptionnew} implies that each conditional
spin agrees with the sign of $h_w$ with probability at least
$1-p_0/4$ (in the original $\{-1,+1\}$ coordinates).
This bound also holds under partial conditioning, by averaging
over the unrevealed spins.
Consequently, the common-uniform coupling can create a
disagreement at $w$ only if
$|h_w|\leq K$ or $\min(U_w,1-U_w)\leq p_0/4$.
Thus declare $u,v$ open and declare each other site $x$ open
on the event
\begin{equation}\label{ateachsite}
\{|h_x|\leq K\}\cup
\{\min(U_x,1-U_x)\leq p_0/4\}.
\end{equation}
Every disagreement outside $\{u,v\}$ is created at an open
site adjacent to an earlier disagreement. By induction, all
disagreements lie in the connected component of $\{u,v\}$
in this site percolation, with $u,v$ forced open.
The final identity coupling creates no further disagreements.

Here the variables $(U_x)_{x\in V}$ are chosen i.i.d. uniform
on $[0,1]$, independently of $h$.
Under the joint law of $(h,U)$, the site indicators defined
by \eqref{ateachsite} are independent and satisfy
\[
\mathbb{P}_{h,U}(x\text{ is open})
\leq \mathbb{P}_h(|h_x|\leq K)+\frac{p_0}{2}
< p_0.
\]
Hence this percolation is stochastically dominated by
independent Bernoulli site percolation with parameter $p_0$.
The domination is preserved when $u$ and $v$ are forced
open in both processes.

By definition of total variation distance, we have 
    $$
|\operatorname{Cor}_\nu^{(2)}(uv,wz)|\leq \mathbb{P}(\sigma_{wz}^1\neq\sigma_{wz}^2)\leq\mathbb{P}(\sigma_w^1\neq\sigma_w^2)+\mathbb{P}(\sigma_z^1\neq\sigma_z^2)
    $$ for our specific coupling (here $\sigma_{wz}^1\neq\sigma_{wz}^2$ means $\sigma_w^1\neq\sigma_w^2$ or $\sigma_z^1\neq\sigma_z^2$), and thus since each vertex is adjacent to at most $\Delta$ edges, we multiply by $\Delta$ and write
\[
\sum_{(wz)\in E}
|\operatorname{Cor}_\nu^{(2)}(uv,wz)|
\leq \Delta\,\mathbb{E}_U
\bigl[|\mathcal{C}_{h,U}(u,v)|\bigr],
\]
where $\mathcal{C}_{h,U}(u,v)$ denotes the connected
component of $(u,v)$ in the percolation defined by
\eqref{ateachsite}, with $u$ and $v$ forced open.
The expectation $\mathbb{E}_U$ is taken over $U$
with $h$ fixed.
Denote by $S_{uv}(h):=\frac{1}{\Delta}\sum_{(wz)\in E}|\operatorname{Cor}_\nu^{(2)}(uv,wz)|$ (This notation has implicit dependence on $\theta,\Lambda,\tau_0$). Then for any $s>0$,
$$
\mathbb{P}_h(S_{uv}(h)\geq m)\leq e^{-sm}\mathbb{E}_{h}\exp(s\mathbb{E}_U|\mathcal{C}_{h,U}|)\leq e^{-sm}\mathbb{E}_{h,U}\exp(s|\mathcal{C}_{h,U}|),
$$applying Jensen in the last step.
    
    Moreover, for any $\theta\in[0,\theta_*]$, any $\Lambda$ and any boundary condition $\tau_0$, we use the same coupling for $\nu$ and $\nu(\cdot\mid uv)$ via the same uniform distribution, and thus the discrepancy region is simultaneously controlled by the same Bernoulli percolation \eqref{ateachsite}. 
That is,
$$
\mathbb{P}_h(\sup_{\theta\in[0,\theta_*],\Lambda,\tau_0}S_{uv}(h)\geq m)\leq e^{-sm}\mathbb{E}_{h,U}\exp(s|\mathcal{C}_{h,U}|).
$$
By the stochastic domination established above, for $s>0$
the exponential moment of $|\mathcal{C}_{h,U}(u,v)|$
under the joint law of $(h,U)$ is bounded by the corresponding
moment of the independent Bernoulli site percolation cluster
with parameter $p_0$, with $u$ and $v$ forced open.
Taking $s=\alpha_*$ and applying \eqref{lines7000}
to this dominating cluster yields \eqref{lines6750}. 
    
    The rest of the proof is the derivation of \eqref{lines7000} for the exponential moments for the size of the connected component. We consider a Bernoulli $(p_0)$-percolation on $G$ whose law is denoted by $\mathbb{P}$, and we use $\mathcal{C}(u,v)$ denote the connected cluster in this percolation where $(u,v)$ are forced open. Then we explore the open cluster of the edge $(u,v)$. Since the maximal degree is at most $\Delta$ and $(u,v)$ are forced open, the connected cluster is stochastically bounded by a Galton-Watson forest with two initial particles and offspring distribution $\operatorname{Bin}(\Delta-1,p_0)$, see for example \cite{Grimmett1999Percolation}. By the Otter-Dwass formula \cite{Dwass1969TotalProgeny} (and see \cite{Pitman2006CSP} for the $r=2$ version), if we let $T^{(2)}$ denote the total progeny of the forest, then 
$$
\mathbb{P}(T^{(2)}=x)=\frac{2}{x}\mathbb{P}(\operatorname{Bin}((\Delta-1)x,p_0)=x-2),\quad \forall x\in\mathbb{N}_+.
$$ By Chernoff, whenever $p_0(\Delta-1)<1$, we have
$$
\mathbb{P}(T^{(2)}=x)\leq\frac{2}{x}(\frac{1-p_0}{p_0(\Delta-2)})^2e^{-\xi_*x},
$$where we take 
$\xi_*=\log\frac{(\Delta-2)^{\Delta-2}}{p_0(\Delta-1)^{\Delta-1}(1-p_0)^{\Delta-2}}.$ Then we sum up over integers $x\geq m$ and get 
\begin{equation}\label{lines7000none}
\mathbb{P}(|\mathcal{C}(u,v)|\geq m\mid (u,v)\text{ open })\leq \frac{2}{1-e^{-\xi_*}}(\frac{1-p_0}{p_0(\Delta-2)})^2e^{-\xi_* m}.
\end{equation}Then we can compute the exponential moment of $\mathcal{C}(u,v)$: take $\alpha_*=\frac{1}{2}\xi_*$, then
\begin{equation}\label{lines7000}
\mathbb{E}[\exp(\alpha_*|\mathcal{C}(u,v)|)\mid (u,v)\text{ open }]\leq \frac{2}{(1-e^{-2\alpha_*})(1-e^{-\alpha_*})}(\frac{1-p_0}{p_0(\Delta-2)})^2.
\end{equation}
\end{proof}

The row-sum estimate directly yields a uniform bound on the largest
eigenvalue of the correlation matrix.
\begin{Proposition}\label{proposition672newtheorem}
Under the assumptions of Proposition \ref{proposition672},
for every $m\geq1$,
\begin{equation}\label{lines6750new}
\begin{aligned}
&\mathbb{P}_h\left(
\sup_{\theta\in[0,\theta_*],\Lambda,\tau_0\in\Omega(\Lambda)}
\lambda_{\max}\!\left(
\operatorname{Cor}^{(2)}_{(1-\theta)\otimes\mu^{\tau_0}}
\right)\geq\Delta m\right)\\
&\qquad\leq
|E|\,\frac{2e^{2\gamma_*}}
{(1-e^{-\alpha_*})(1-e^{-2\alpha_*})}
e^{-\alpha_*m}.
\end{aligned}
\end{equation}
\end{Proposition}

\begin{proof}
The eigenvalues of $\operatorname{Cor}^{(2)}_\nu$ are real, and
every eigenvalue of a matrix is bounded in absolute value by its
maximum absolute row sum. Hence
\[
\lambda_{\max}(\operatorname{Cor}^{(2)}_\nu)
\leq\max_{e\in E}\sum_{f\in E}
|\operatorname{Cor}^{(2)}_\nu(e,f)|.
\]
The claim follows by taking a union bound in Proposition
\ref{proposition672} over the initial edge $e$.
If $E=\varnothing$, the correlation matrix is empty and we take
the spectral bound to be zero, so the claim is immediate.
\end{proof}

Choosing $m$ of order $\log|V|$ gives the following consequence.
\begin{Proposition}\label{proposition4.3}
Under the assumptions of Proposition \ref{proposition672},
with probability at least $1-O_{\beta,\Delta,p_0}(1/|V|)$
over the external field $h$,
\[
\sup_{\theta\in[0,\theta_*],\Lambda,\tau_0\in\Omega(\Lambda)}
\lambda_{\max}\!\left(
\operatorname{Cor}^{(2)}_{(1-\theta)\otimes\mu^{\tau_0}}
\right)
\leq\frac{4\Delta\log|V|}{\alpha_*}.
\]
\end{Proposition}

\begin{proof}
Write $n=|V|$. For $n\geq e^{\alpha_*/2}$, take
$m=2\log n/\alpha_*\geq1$ in \eqref{lines6750new}.
Since $|E|\leq n\Delta/2$, the exceptional probability is at most
\[
\frac{\Delta e^{2\gamma_*}}
{(1-e^{-\alpha_*})(1-e^{-2\alpha_*})}\,\frac1n.
\]
On the complementary event, uniformly over $\theta,\Lambda,\tau_0$,
\[
\lambda_{\max}\!\left(
\operatorname{Cor}^{(2)}_{(1-\theta)\otimes\mu^{\tau_0}}
\right)
\leq\frac{2\Delta\log n}{\alpha_*}
\leq\frac{4\Delta\log n}{\alpha_*}.
\]
For $n<e^{\alpha_*/2}$, the stated probability bound follows by
enlarging its implicit constant to at least $e^{\alpha_*/2}$.
\end{proof}

\begin{corollary}\label{2corollary44} 
    Under the same assumption in Proposition \ref{proposition672}, with $\mathbb{P}_h$-probability at least $1-O_{\beta,\Delta,p_0}(\frac{1}{|V|})$, the edge field denoising process $(Y_t)_{t\in[0,1]}$ satisfies $R$- approximate conservation of variance up to time $\theta_*$ with constant 
    $$
R\leq\exp\left(\frac{16\beta\Delta\ln|V|}{\alpha_*}\right).
    $$
    Moreover, the RFIM measure $\mu_G$ satisfies approximate tensorization of variance with constant $R$.
\end{corollary}
\begin{proof} Note that $\int_0^{\theta_*}\frac{dt}{1-t}=4\beta$.
    The claim on $(Y_t)$ follows directly from Proposition \ref{proposition4.3}, Lemma \ref{lemma3760} and
Corollary \ref{followingcriterion}. The $R$-approximate tensorization of variance of $\mu_G$ then follows from Lemma \ref{lemma3500}, since the posterior $\operatorname{Law(}Y_1\mid Y_{\theta_*})$ is a product measure. \end{proof}

Now we can conclude the proof of Theorem \ref{theorem12121}:

\begin{proof}[\proofname\ of Theorem \ref{theorem12121}] The claim on Poincaré constant follows from Fact \ref{fact3900} combined with Corollary
    \ref{2corollary44}.
The claim on mixing times follows from Fact \ref{firststandardfact}.
\end{proof}

\section{Approximate tensorization of entropy and MLSI}
This section proves Theorem \ref{mlsitheorembound} on MLSI for RFIM. We first recall an equivalent condition of entropic stability from \cite{chen2026edge}, which is the entropic analogue of Lemma \ref{lemma88}:

\begin{Definition}
    Consider a distribution $\mu$ on $\{0,1\}^m$. Then we say that $\mu$ is $C$-entropically stable if for any absolutely continuous $\nu\ll\mu$,
$$
\sum_{i=1}^m\left[\nu(i)\log\frac{\nu(i)}{\mu(i)}-(\nu(i)-\mu(i))\right]\leq CD_{KL}(\nu\mid\mid\mu).
$$
Here $\nu(i):=\nu(\sigma_i=1)$ and $\mu(i):=\mu(\sigma_i=1)$.
We let $\operatorname{ES}(\mu)$ denote the minimal $C>0$ such that $\mu$ is $C$-entropically stable.
\end{Definition}
Assuming a lower bound on the marginal density, entropic stability can be deduced from spectral stability for the vertex field denoising process:
\begin{lemma}\label{lemma4.20}(\cite{chen2026edge}, Lemma 7.15)
    Let $\theta\in(0,1)$ and given constants $\eta_{\operatorname{op}}\geq 1,K_{\operatorname{low}}\geq 1$. Let $\mu$ be a probability distribution on $\{0,1\}^m$. Assume that for any $t\in[0,\theta]$ and every feasible pinning $\tau=\mathbf{1}_S$, where $S\subseteq [m]$, the distribution $\nu:=(1-t)*\mu^\tau$ satisfies
    \begin{itemize}
        \item Marginal bound: $\nu(i)\geq\frac{1}{K_{\operatorname{low}}}$ for all $i\in[m]$;
        \item Spectral stability: $\lambda_{max}(\operatorname{Cor}_\nu)\leq\eta_{\operatorname{op}}$,
    \end{itemize}
    Then if we let $L:=(K_{\operatorname{low}}+1)(\eta_{\operatorname{op}}-1)+1$, then $\mu$ satisfies
$$
\operatorname{ES}(\mu)\leq\frac{L}{1-(1-\theta)^L}
.$$ Meanwhile, the vertex-field denoising process $(Y_t)_{t\in[0,1]}$ from $\mu$ satisfies $R$-approximate conservation of entropy up to final time $\theta$ with 
$$
R=1+\frac{\max_\tau\operatorname{ES}((1-\theta)*\mu^\tau)}{L(1-\theta)^L}
.$$
\end{lemma}

We apply this lemma to the vertex field denoising process $Y_t'$ where the vertices are the subsets $A_{uv}:=\{\sigma_u=\sigma_v=1\}$ for each $(u,v)\in E$. Then by Lemma \ref{lemma2.16}, this translates to approximate conservation of entropy of the edge-field denoising process $Y_t$.

It is easy to check that the marginals are lower bounded when external fields are bounded:

\begin{lemma}
Under the bounded field assumption \eqref{boundexternalfield}, for any edge $(u,v)\in E$, the following holds for any $t\in[0,\theta_*]$, any $S\subset\{A_{(x,y)},(x,y)\in E\}$ and $\tau=\mathbf{1}_S$: take $\nu:=(1-t)\otimes \mu^\tau$ (or equivalently on the event space the marginal is $(1-t)*\pi_{\mathcal{A}}^\tau$ ), then
    \begin{equation}\label{}
\nu(A_{uv})\geq \left(\frac{e^{-\beta\Delta-M}}{e^{-\beta\Delta-M}+e^{\beta\Delta+M}}\right)^2:=\frac{1}{C_{\Delta,M,\beta}},
    \end{equation}  
so that we define \begin{equation}\label{cemmdelta}C_{\Delta,M,\beta}:=\left(\frac{e^{-\beta\Delta-M}+e^{\beta\Delta+M}}{e^{-\beta\Delta-M}}\right)^2.\end{equation}
\end{lemma}Here the marginal of the event-coordinate measure coincides with the probability of the events $A_{uv}$ under the tilted measure $(1-t)\otimes\mu^\tau$.
\begin{proof} Since $t\in[0,\theta_*]$, the quadratic coefficients in $\nu$ in the coordinate $\{0,1\}^V$ are in $[0,4\beta]$, and in this coordinate the external fields take their value in $[-2M-2\beta\Delta,2M]$ at each vertex by Fact \ref{fact2690}. Then we compute that for any configuration $\tau'$ on $V\setminus\{u\}$ consistent with $\tau$, $\nu(\sigma_u=1\mid \tau')\geq \frac{e^{-2\beta\Delta-2M}}{e^{-2\beta\Delta-2M}+1}=\frac{e^{-\beta\Delta-M}}{e^{-\beta\Delta-M}+e^{\beta\Delta+M}}$. It suffices to apply the conditioning twice. 
\end{proof}

From this, the proof of Theorem \ref{mlsitheorembound} is now immediate:

\begin{proof}[\proofname\ of Theorem \ref{mlsitheorembound}]
Write $n=|V|$. The case $n=1$ is immediate, so assume $n\geq2$.
 Assume that the bound on the largest eigenvalue in Proposition \ref{proposition4.3} holds, which has probability $1-O_{\beta,\Delta,p_0}(\frac{1}{|V|})$.
In Lemma \ref{lemma4.20}, take
$\eta_{\operatorname{op}}=1+4\Delta\log n/\alpha_*$ and
$L=(C_{\Delta,M,\beta}+1)4\Delta\log n/\alpha_*+1$.
For each feasible $\tau$, write
$P_\tau=(1-\theta_*)\otimes\mu^\tau$ for the terminal vertex product measure
and $Q_\tau=(1-\theta_*)*\pi_{\mathcal A}^\tau$ for its edge-event image.
Given $\nu\ll Q_\tau$, lift it to $\widehat\nu$ on the vertices using
the conditional law of $P_\tau$ given the edge indicators.
Data processing and the relative entropy chain rule for the product
measure $P_\tau$, with each vertex occurring in at most $\Delta$ edges, give
\[
\begin{aligned}
\sum_{e\in E}\left[\nu(e)\log\frac{\nu(e)}{Q_\tau(e)}
-\nu(e)+Q_\tau(e)\right]
&\leq\sum_{e\in E}D_{KL}(\widehat\nu_e\mid\mid(P_\tau)_e)\\
&\leq\Delta D_{KL}(\widehat\nu\mid\mid P_\tau)
=\Delta D_{KL}(\nu\mid\mid Q_\tau),
\end{aligned}
\]
where the subscript $e=\{u,v\}$ denotes the two-vertex marginal.
Thus $\max_\tau\operatorname{ES}(Q_\tau)\leq\Delta$.
The assumptions imply
$\alpha_*<\frac12\log(1/p_0)<K<M<C_{\Delta,M,\beta}$,
so $L\geq\Delta$ and hence
$\max_\tau\operatorname{ES}(Q_\tau)/L\leq1$.
We then apply Lemma \ref{lemma4.20} on this event, which shows that the edge field denoising process $Y_t'$ satisfies $R$-entropic conservation of entropy with 
    $$
R\leq 1+2(\exp(4\beta))^{(C_{\Delta,M,\beta}+1)(\frac{4\Delta\ln|V|}{\alpha_*})+1}\leq 3(\exp(4\beta))^{(C_{\Delta,M,\beta}+1)(\frac{4\Delta\ln|V|}{\alpha_*})+1}
    $$
    up to time $\theta_*$. Then by Lemma \ref{lemma2.16}, this yields $R$-conservation of entropy for the denoising process $Y_t$ up to time $\theta_*$. Then since $\operatorname{Law}(Y_1\mid Y_{\theta_*})$ is a product measure, we apply Lemma \ref{lemma3500} to deduce that $\mu_G$  satisfies approximate tensorization of entropy with constant $R$. This implies the claimed MLSI lower bound $\rho_{LS}(P)\geq\frac{1}{R|V|}$ by Fact \ref{fact4160}, which implies the claimed mixing time estimate by Fact \ref{firststandardfact}.
\end{proof}

\section{Fast relaxation under WSM for RFIM on general graphs}

In this section we focus on RFIM satisfying $\operatorname{WSM}(C)$ on a graph $G$ of $\alpha$-stretched-exponential growth. We will first take a similar step as in \cite{el2026fast} where we use a stochastic localization (by Brownian SDEs) to virtually increase the strength of the external field while keeping the quadratic interaction unchanged. This step is achieved via the SDE-driven localization scheme of Chen and Eldan \cite{chen2022localization}, and is fundamentally different from the edge-tilted localization scheme in Section \ref{sections344}.

\subsection{Localization scheme by Brownian SDEs}

Let $\nu_0$ be a probability distribution on the hypercube $\mathcal{C}_{|V|}:=\{-1,+1\}^V$. (In our application, $\nu_0=\mu_G$ is the RFIM measure on $G$.) Let $(B_t)_{t\geq 0}$ be a standard Brownian motion on $\mathbb{R}^{|V|}$ with $B_0=0$. Then $\operatorname{SL}$ is the stochastic localization process $(\nu_t(\sigma))_{t\geq 0}$ such that its density with respect to $\nu_0$
$$
\frac{d\nu_t}{d\nu_0}(\sigma)=F_t(\sigma)
$$
solves the following SDE
$$
\begin{cases}
    dF_t(\sigma)=F_t(\sigma)\langle \sigma-a_t,dB_t\rangle,\quad\forall\sigma\in\mathcal{C}_{|V|},\\F_0(\sigma)=1,
\end{cases}
$$and here $a_t$ is the mean of the sample from $\nu_t$:
\begin{equation}
    a_t=\int_{\mathcal{C}_{|V|}}\sigma \nu_t(d\sigma).
\end{equation}Then almost surely for all $t>0,$ $\nu_t$ is a probability measure and for any $A\subset\mathcal{C}_{|V|}$, the function $t\mapsto \nu_t(A)$ is a martingale.

A benefit of our specific choice of $\operatorname{SL}$ path is that the measure $\nu_t$ admits the following rewriting 
$$
\nu_t(\sigma)\propto e^{\langle y_t,\sigma\rangle}\nu_0(\sigma),
$$
and such that $y_t$ solves the following $\operatorname{SDE}$
$$
dy_t=a(y_t)dt+dB_t,\quad y_0=0,
$$
with $a(y_t)=a_t$.

Indeed, the process $y_t$ in the $\operatorname{SL}$ process admits the following Bayesian interpretation:

\begin{lemma}\label{lemma856}(\cite{el2026fast}, Proposition 2.1) Let $(\bar{B}_t)_{t\geq 0}$ be a standard Brownian motion and let $\sigma^*\sim \nu_0$ independently of $(\bar{B}_t)_{t\geq 0}$. Then the process $(y_t)_{t\geq 0}$ has the same law as $(\bar{y}_t)_{t\geq 0}$ which is defined via 
\begin{equation}
    \bar{y}_t=t\sigma^*+\bar{B}_t,\quad t\geq 0.
\end{equation}The $\operatorname{SL}$ path $(\nu_t)_{t\geq 0}$ is distributed as a process of conditional measures: $(\nu_0(\sigma^*\in\cdot\mid \bar{y}_t))_{t\geq 0}$.
    Therefore, $\nu_t$ is an RFIM measure obtained from $\nu_0=\mu_G$ by adding $\bar{y}_t$ to the existing external field $h$. This virtually increases the strength of the external field of the RFIM. 
\end{lemma}

\subsection{Weak Poincaré inequality}

We prove the following moment estimates for the correlation functions of $\nu_t$:

\begin{theorem}\label{theorem83777}
    Suppose that the graph $G=(V,E)$ has $\alpha$-stretched exponential growth for some $\alpha\in(0,1)$. Suppose that the RFIM measure $\mu_G$ on $G$ satisfies $\operatorname{WSM}(C)$ for some $C>0$. Then we can find $C_0=C_0(\alpha,C_\alpha,C)>0$ such that for any integer $p\geq 1$,
$$
\mathbb{E}_h\left[\sup_{t\geq 0}\mathbb{E}[\operatorname{Tr}(\operatorname{Cov}(\nu_t)^p)\mid h]\right]\leq (C_0p)^p|V|
.$$
    In particular, for any $R\geq 0$, by Markov's inequality, 
    $$
\mathbb{E}_h\left[\sup_{t\geq 0}\mathbb{P}\left(\|\operatorname{Cov}(\nu_t)\|_{op}\geq R\mid h\right)\right]\leq e\,|V|e^{-c_0R},
    $$with $c_0=(eC_0)^{-1}$.
\end{theorem}

Theorem \ref{theorem83777} is proven at the end of Section \ref{section666}.

Then we prove a weak version of the variance concentration result:

\begin{theorem}\label{theorem5.345}
    Suppose the graph $G$ has $\alpha$-stretched-exponential growth and the RFIM $\mu_G$ satisfies $\operatorname{WSM}(C)$ for some $C>0$. Then we can find some $c_0=c_0(\alpha,C_\alpha,C)>0$ such that for any $\delta>0$, with $\mathbb{P}_h$-probability at least $1-\delta$ the following estimate holds. For any $T>0$ and test function $\varphi:\{-1,+1\}^V\to\mathbb{R}$, there holds
\begin{equation}\label{relationattheorem5.3}
    \operatorname{Var}_{\nu_0}(\varphi)\leq (e^{-c_0}|V|/\delta)^{1/q}\mathbb{E}[\operatorname{Var}_{\nu_T}(
    \varphi)\mid h]^{\frac{1}{p}}\operatorname{osc}(\varphi)^{\frac{2}{q}},
\end{equation}
    where $p=e^{2T/c_0}$ and $q$ satisfies $1/p+1/q=1$. The expectation is taken with respect to the $\operatorname{SL}$ process $(\nu_t)_{t\geq 0}$ and we denote by $\operatorname{osc}(
    \varphi):=\sup_{x\in\{-1,+1\}^V}\varphi(x)-\inf_{x\in\{-1,+1\}^V}\varphi(x)$.
\end{theorem}

\begin{proof}
We follow the variance comparison argument of
\cite{el2026fast}, Theorem 3.1, first constructing a single
event on which the required tail bounds hold uniformly
in the threshold. Write $n=|V|$ and
$A_t=\operatorname{Cov}(\nu_t)$. Let $C_0\geq1$ be a constant for which the moment
bound in Theorem \ref{theorem83777} holds, and set
\[
c_0=(2eC_0)^{-1},\qquad
Z(h):=\sup_{t\geq0}
\mathbb{E}\bigl[e^{c_0\|A_t\|_{op}}\mid h\bigr].
\]
Since $A_t$ is positive semidefinite,
$\|A_t\|_{op}^m\leq\operatorname{Tr}(A_t^m)$ for every
integer $m\geq1$. Expanding the exponential and applying
the moment bound in Theorem \ref{theorem83777}, we obtain
\[
\mathbb{E}_h Z(h)
\leq 1+n\sum_{m=1}^{\infty}
\frac{c_0^m(C_0m)^m}{m!}
\leq 1+n\sum_{m=1}^{\infty}2^{-m}
\leq 2n,
\]
where we used $m!\geq(m/e)^m$.
Hence Markov's inequality gives
\[
\mathbb{P}_h(\mathcal H_\delta)\geq1-\delta,
\qquad
\mathcal H_\delta:=\{h:Z(h)\leq2n/\delta\}.
\]
For every $h\in\mathcal H_\delta$, exponential Markov
inequality yields
\begin{equation}\label{suppoppps}
\sup_{t\geq0}
\mathbb{P}\bigl(\|A_t\|_{op}\geq R\mid h\bigr)
\leq \frac{2n}{\delta}e^{-c_0R},
\qquad \forall R\geq0.
\end{equation}
We fix $h\in\mathcal H_\delta$ throughout the remainder
of the proof. All expectations below are conditional
on $h$ and taken over the stochastic localization process.

Consider the martingale $M_t:=\nu_t(\varphi)$, then 
$$
dM_t=\int_{\{-1,+1\}^V}\varphi(x)\langle x-a_t,dB_t\rangle\nu_t(dx)
,$$and thus
$$
d[M]_t=\|\int_{\{-1,+1\}^V}\varphi(x)(x-a_t)\nu_t(dx)\|^2dt\leq\operatorname{Var}_{\nu_t}(\varphi)\|\operatorname{Cov}(\nu_t)\|_{op}dt,
$$where we recall that 
$$
\operatorname{Cov}(\nu_t):=\int_{\{-1,+1\}^V}(\sigma-a_t)(\sigma-a_t)^T\nu_t(d\sigma),\quad a_t=\int_{\{-1,+1\}^V}\sigma \nu_t(d\sigma).
$$
Next, by the fact that $\nu_t(\varphi^2)$ is a martingale and Itô's formula, 
\begin{equation}
    d\operatorname{Var}_{\nu_t}(\varphi)=d\nu_t(\varphi^2)-2M_tdM_t-\frac{1}{2}2d[M]_t=-d[M]_t+\operatorname{martingale}.
\end{equation}
We may assume that $0<\delta<1$ and that $\varphi$ is
nonconstant, since the other cases are immediate. Set
\[
S:=\frac14\operatorname{osc}(\varphi)^2,\qquad
F(t):=\mathbb{E}[\operatorname{Var}_{\nu_t}(\varphi)\mid h],
\qquad B:=\frac{2n}{\delta}.
\]
The elementary variance bound
$\operatorname{Var}_{\nu_t}(\varphi)\leq S$ follows by
centering $\varphi$ at the midpoint of its range.
Thus $0<F(t)\leq S$ for every finite $t$. By
\eqref{suppoppps}, for every $R\geq0$,
\[
\begin{aligned}
&\mathbb{E}\bigl[\|A_t\|_{op}
\mathbf{1}_{\{\|A_t\|_{op}\geq R\}}\mid h\bigr]\\
&\quad=R\mathbb{P}(\|A_t\|_{op}\geq R\mid h)
+\int_R^\infty\mathbb{P}(\|A_t\|_{op}\geq r\mid h)\,dr\\
&\quad\leq B(R+c_0^{-1})e^{-c_0R}.
\end{aligned}
\]
The preceding quadratic variation estimate therefore gives
\[
F'(t)\geq-RF(t)-SB(R+c_0^{-1})e^{-c_0R}.
\]
Since \eqref{suppoppps} holds for every threshold on the
same event $\mathcal H_\delta$, we may choose
\[
R=c_0^{-1}\log\frac{BS}{F(t)}\geq0.
\]
It follows that
\[
F'(t)\geq-c_0^{-1}
\left(2\log\frac{BS}{F(t)}+1\right)F(t).
\]
Writing $u(t):=F(t)/S$, we obtain
\[
\frac{d}{dt}\log u(t)
\geq\frac{2}{c_0}\log u(t)
-\frac{2\log B+1}{c_0}.
\]
Multiplying by $e^{-2t/c_0}$ and integrating from $0$ to $T$
yields
\[
\log u(0)\leq e^{-2T/c_0}\log u(T)
+(1-e^{-2T/c_0})(\log B+1/2).
\]
Consequently, with $p=e^{2T/c_0}$ and $1/p+1/q=1$,
\[
\operatorname{Var}_{\nu_0}(\varphi)
\leq\left(\frac{\sqrt e\,n}{2\delta}\right)^{1/q}
\mathbb{E}[\operatorname{Var}_{\nu_T}(\varphi)\mid h]^{1/p}
\operatorname{osc}(\varphi)^{2/q}.
\]
Since $C_0\geq1$, we have
$c_0\leq(2e)^{-1}<\log2-1/2$, and hence
$\sqrt e/2\leq e^{-c_0}$. This proves
\eqref{relationattheorem5.3}.
The event $\mathcal H_\delta$ is independent of $T$ and
$\varphi$, so the conclusion holds simultaneously for all
$T>0$ and all test functions.
\end{proof}

Then we can prove the following weak Poincaré inequality for RFIM:
\begin{theorem}\label{statementtheorem906}
    Suppose $G$ has $\alpha$-stretched-exponential growth for some $\alpha\in(0,1)$ and $\operatorname{WSM}(C)$ holds for RFIM $\mu_G$ for some $C>0$. Then for any fixed $\delta\in(0,1)$, with $\mathbb{P}_h$-probability at least $1-\delta$, we have the following inequality for all functions $\varphi:\{-1,+1\}^V\to\mathbb{R}$:
    \begin{equation}
\operatorname{Var}_{\mu_G}(\varphi)\leq K\delta^{-1/q}|V|^\kappa\mathcal{E}_{\mu_G}(\varphi,\varphi)^{1/p}\cdot\operatorname{osc}(\varphi)^{2/q},        
    \end{equation}
    where the constants $K,\kappa>0$ and $p,q\geq 1,1/p+1/q=1$ all depend on $C,C_\alpha,\alpha$ and $\beta$.
\end{theorem}

To prove Theorem \ref{statementtheorem906}, we need the following refinement of Theorem \ref{theorem12121}:

\begin{theorem}\label{refinementtheorem}
For any $\epsilon>0$, $\beta>0$ and $C_\alpha>0$, there exist constants $K_0>0$, $\kappa_0$ and $\epsilon_0>0$ (depending only on $\epsilon,\beta,C_\alpha$) such that the following holds. Let $G=(V,E)$ be a graph of maximal degree $\Delta\leq\exp(C_\alpha)$, and $\mu_G$ be the RFIM on $G$ with external field $h$ where $(|h_u|)_{u\in V}$ are independent,
not necessarily identically distributed, and satisfy
$\sup_{u\in V}\mathbb{P}_h(|h_u|\leq K_0)\leq\epsilon_0$. Then for each $L\geq\kappa_0\ln|V|$, 
$$
\mathbf{gap}_G^{-1}\leq|V|\cdot\exp( {\epsilon} L)
$$with $\mathbb{P}_h$- probability at least $1-e^{-2L}$. 
\end{theorem}

The proof of Theorem \ref{refinementtheorem} is presented after the proof of Theorem \ref{statementtheorem906}.

\begin{proof}[\proofname\ of Theorem \ref{statementtheorem906}]
Write $n=|V|$ and fix $\eta=1/2$. Apply Theorem
\ref{refinementtheorem} with $\epsilon=\eta$ to obtain constants
$K_0,\kappa_0,\epsilon_0>0$ depending only on $\beta$ and $C_\alpha$.
Choose
\[
T\geq\max\left\{1,\frac{2K_0^2}{\pi\epsilon_0^2}\right\},
\]
with $T$ depending only on these constants.

We use the Bayesian representation in Lemma \ref{lemma856}:
conditional on $h$, sample $\sigma^*\sim\mu_G$, and let
$(\bar B_t)_{t\geq0}$ be a standard Brownian motion independent
of $(h,\sigma^*)$. In this representation, $\nu_T$ is the RFIM
with external field
\[
H_T=h+T\sigma^*+\bar B_T.
\]
Conditional on $(h,\sigma^*)$, the coordinates of $H_T$ are
independent Gaussian variables with respective means
$h_v+T\sigma_v^*$ and common variance $T$. Their absolute values
are therefore independent, although they need not have the same
distribution. The Gaussian density bound gives
\begin{equation}\label{intermediateprobest}
\sup_{v\in V}\mathbb{P}\bigl(|H_T(v)|\leq K_0\mid h,\sigma^*\bigr)
\leq\frac{2K_0}{\sqrt{2\pi T}}\leq\epsilon_0.
\end{equation}
This estimate is uniform in $h$ and $\sigma^*$ and requires no
symmetry assumption on $h$.

For $L\geq\kappa_0\log n$, let
\[
E_L:=\{\mathbf{gap}_{\nu_T}^{-1}\leq ne^{\eta L}\}.
\]
Applying Theorem \ref{refinementtheorem} under the conditional
law given $(h,\sigma^*)$, we obtain
\[
\mathbb{P}(E_L^c\mid h,\sigma^*)\leq e^{-2L}.
\]
Averaging over $\sigma^*$ conditional on $h$ yields
\[
\mathbb{P}(E_L^c\mid h)
=\mathbb{E}\bigl[\mathbb{P}(E_L^c\mid h,\sigma^*)\mid h\bigr]
\leq e^{-2L}.
\]
Thus, for every fixed $h$, this tail bound holds for all
$L\geq\kappa_0\log n$. On $E_L$, the Poincar\'e inequality gives
\begin{equation}\label{firstpoincareeq}
\operatorname{Var}_{\nu_T}(\varphi)
\leq ne^{\eta L}\mathcal{E}_{\nu_T}(\varphi,\varphi)
\quad\text{for all }\varphi:\{-1,+1\}^V\to\mathbb{R}.
\end{equation}

Fix $h$ and a test function $\varphi$, which may depend on $h$
but not on the localization randomness, and write
\[
D:=\mathcal{E}_{\mu_G}(\varphi,\varphi),\qquad
R:=\operatorname{osc}(\varphi)^2,\qquad
F:=\mathbb{E}[\operatorname{Var}_{\nu_T}(\varphi)\mid h].
\]
The posterior identity
$\mathbb{E}[\nu_T(\sigma)\mid h]=\mu_G(\sigma)$ and the joint
concavity of $(a,b)\mapsto ab/(a+b)$ imply, by applying Jensen's
inequality to each term in the Dirichlet form, that
\[
\mathbb{E}[\mathcal{E}_{\nu_T}(\varphi,\varphi)\mid h]\leq D;
\]
see also \cite{eldan2022spectral}, Lemma 9. Here the expectation
averages over both $\sigma^*$ and $\bar B_T$, with $h$ fixed.
Splitting according to $E_L$ and using
$\operatorname{Var}_{\nu_T}(\varphi)\leq R$, we obtain
\begin{equation}\label{whatisbutnoton}
\begin{aligned}
F&\leq ne^{\eta L}\mathbb{E}
[\mathcal{E}_{\nu_T}(\varphi,\varphi)\mathbf{1}_{E_L}\mid h]
+R\mathbb{P}(E_L^c\mid h)\\
&\leq ne^{\eta L}D+e^{-2L}R.
\end{aligned}
\end{equation}
This holds for every fixed $h$ and every $L\geq\kappa_0\log n$,
so $L$ may be chosen as a function of $h$ and $\varphi$.

The conclusion is immediate for constant $\varphi$. Otherwise,
$D,R>0$. Set $L_0:=\kappa_0\log n+1$. If $D\leq e^{-L_0}R$,
choose $L=\log(R/D)\geq L_0$ in \eqref{whatisbutnoton}. Since
$D\leq R$ and $n\geq1$, this gives
\[
F\leq nD^{1-\eta}R^\eta+\frac{D^2}{R}
\leq2nD^{1-\eta}R^\eta.
\]
By Theorem \ref{theorem5.345}, on a single event of
$\mathbb{P}_h$-probability at least $1-\delta$, for every
$\varphi$ we have
\[
\operatorname{Var}_{\mu_G}(\varphi)
\leq(n/\delta)^{1/s}F^{1/r}R^{1/s},
\qquad r=e^{2T/c_0},\qquad \frac1r+\frac1s=1.
\]
Consequently, on this event, in the case $D\leq e^{-L_0}R$,
\begin{equation}\label{secondvarphis}
\operatorname{Var}_{\mu_G}(\varphi)
\leq2^{1/r}n\delta^{-1/s}
D^{(1-\eta)/r}R^{\eta/r+1/s}.
\end{equation}
Define $p,q>1$ by
\[
\frac1p=\frac{1-\eta}{r},\qquad
\frac1q=\frac{\eta}{r}+\frac1s.
\]
Then $1/p+1/q=1$, and $\delta^{-1/s}\leq\delta^{-1/q}$
because $0<\delta<1$.

In the remaining case $D>e^{-L_0}R$, the elementary bound
$\operatorname{Var}_{\mu_G}(\varphi)\leq R$ gives
\begin{equation}\label{firstvarphis}
\operatorname{Var}_{\mu_G}(\varphi)
\leq R\leq e^{L_0/p}D^{1/p}R^{1/q}
=e^{1/p}n^{\kappa_0/p}D^{1/p}R^{1/q}.
\end{equation}
Combining \eqref{secondvarphis} and \eqref{firstvarphis}, and
setting $\kappa=\max\{1,\kappa_0/p\}$, we obtain, on the same
event of probability at least $1-\delta$,
\[
\operatorname{Var}_{\mu_G}(\varphi)
\leq K\delta^{-1/q}n^\kappa
\mathcal{E}_{\mu_G}(\varphi,\varphi)^{1/p}
\operatorname{osc}(\varphi)^{2/q}
\quad\text{for all }\varphi.
\]
The constants $K,\kappa,p,q$ depend only on
$C,C_\alpha,\alpha$ and $\beta$, as required.
\end{proof}

Finally we complete the proof of Theorem \ref{refinementtheorem}:

\begin{proof}[\proofname\ of Theorem \ref{refinementtheorem}]
The case $|V|=1$ is immediate since $\mathbf{gap}_G=1$.
For $|V|\geq2$, use $\Delta=\max\{3,\lceil e^{C_\alpha}\rceil\}$
as a degree bound throughout this proof.
    We simply take $m=L/\sqrt{\alpha_*}$ in Proposition \ref{proposition672}. Then by the same reasoning in Proposition \ref{proposition4.3} and Corollary \ref{2corollary44}, we deduce that with probability at least $1-|V|\Delta\frac{2}{(1-e^{-\alpha_*})(1-e^{-2\alpha_*})}e^{2\gamma_*}e^{-\sqrt{\alpha_*}\cdot L}$, $\mu_G$ has spectral gap at least 
    $$\mathbf{gap}_G^{-1}\leq |V|\exp(\frac{ 16\beta\Delta}{\sqrt{\alpha_*}}L).
$$
Since $\alpha_*\uparrow\infty$ as $p_0\downarrow 0$, we can find a sufficiently small $p_0>0$ (with respect to the $\epsilon>0$) such that $\frac{16\beta\Delta}{\sqrt{\alpha_*}}\leq\epsilon$ and $\alpha_*\geq 100$.
Choose $K_0$ so that $\rho(\beta,\Delta,K_0)<p_0/4$ and set
$\epsilon_0=p_0/4$; the field assumption then implies
Assumption \ref{thiasssumptionnew}.
Then for this $p_0>0$ we can find $\kappa_0>0$ so that whenever $L\geq\kappa_0\ln|V|$, we have $L/\sqrt{\alpha_*}\geq1$ and $|V|\Delta\frac{2}{(1-e^{-\alpha_*})(1-e^{-2\alpha_*})}e^{2\gamma_*}e^{-\sqrt{\alpha_*}L}\leq e^{-2L}$, so that with probability at least $1-e^{-2L}$, we have $\mathbf{gap}_G^{-1}\leq|V|\cdot\exp(\epsilon L)$ holds. 
\end{proof}

In Theorem \ref{statementtheorem906}, if we take $\varphi=\mathbf{1}_S$ for some $S\subset\{-1,+1\}^V$, we get
$$
\mu_G(S)(1-\mu_G(S))\leq AQ(S,S^c)^{1/p},
$$
where $Q(S,S^c)=\sum_{x\in S,y\in S^c}\mu_G(x)P(x,y)$ and $P$ is the transition probability of Glauber dynamics. 
To construct a sampling algorithm from this estimate, we need the following convergence rate analysis via a weak conductance bound:

\begin{lemma}\label{warmstarttheorem}(\cite{el2026fast}, Lemma 3.5)
    Consider a Glauber dynamics with discrete-time Markov chain $(\sigma_k)_{k\geq 0}$. Let $\pi_k$ be the distribution of $\sigma_k$ and $\pi_\infty=\mu_G$ be the invariant measure. Suppose that $\pi_\infty$ satisfies the following weak conductance bound
    \begin{equation}
\pi_\infty(S)(1-\pi_\infty(S))\leq AQ(S,S^c)^{1/p},\quad\forall S\subseteq \{-1,1\}^V,        
    \end{equation}
    where $p\geq 1,A^p\geq 2/4^{p-1}$, and the following warm start condition holds:
    \begin{equation}
    \max_{S\subseteq \{-1,+1\}^V}\frac{\pi_0(S)}{\pi_\infty(S)}\leq M,
    \end{equation} then 
    \begin{equation}
 d_{\operatorname{TV}}(\pi_k,\pi_\infty)\leq M(\frac{A^{2p}\log k}{k})^{1/(2p-1)}       ,\quad k\geq 4.
    \end{equation}
\end{lemma}

\subsection{The sampling algorithm and its convergence}

In this section, we define the sampling algorithm claimed in
Theorem \ref{theorem231230}. The sampling algorithm is almost
the same as the one given in \cite{el2026fast}, Theorem 1.4.

We use the ordering $v_1,\ldots,v_{|V|}$ given in
Theorem \ref{theorem231230}, and let
$G^{(i)}:=G[\{v_1,\ldots,v_i\}]$.
The RFIM measures $\mu_{G^{(i)}}$ have free boundary conditions
and satisfy $\operatorname{WSM}(C)$ with the same constant $C$.
This uniformity allows us to apply the weak Poincaré inequality
of Theorem \ref{statementtheorem906} at each stage with the same
constants; that theorem itself only requires WSM on the graph
to which it is applied.
The algorithm works in the following three steps:

\begin{enumerate}
    \item Begin with a perfect sample $X_*^{(1)}$ of the RFIM on $G^{(1)}=\{v_1\}$, so that $X_*^{(1)}\in\{-1,+1\}$ is drawn with probability proportional to $e^{\pm h_{v_1}}$.
    
    \item For each $i\geq 2$, we consider a Glauber dynamics $(X_k^{(i)})_{k\geq 0}$ on the RFIM measure $\mu_{G^{(i)}}$ on  $G^{(i)}$, which is initialized from $X_0^{(i)}=[X_*^{(i-1)},\sigma_{v_i}]\in\{-1,1\}^{G^{(i)}}$ (the concatenation of $X_*^{(i-1)}$ with a spin $\sigma_{v_i}\in\{\pm 1\}$ drawn independently with probability proportional to $e^{\pm h_{v_i}}$.)
    \item Then run Glauber dynamics from this initialization in time $k_*=|V|^{C_*}$ and obtain $X_*^{(i)}=X_{k_*}^{(i)}$. 
\end{enumerate}
The constant $C_*$ is fixed in the next Proposition in a way that depends on the error $\epsilon$ and the probability $\delta$. 
\begin{Proposition}
Assume the hypotheses of Theorem \ref{theorem231230},
and use the ordering of vertices
$v_1,v_2,v_3,\ldots,v_{|V|}$ specified there.
With $\mathbb{P}_h$-probability at least $1-\delta$,
the algorithm defined in steps (1)-(3) above produces a sample that is $\epsilon$-close in total variation distance to $\mu_G$, and the algorithm runs in time that is polynomial in $|V|$. Moreover, this can be achieved for any $\epsilon,\delta\geq1/\operatorname{poly}(|V|)$.
This completes the proof of Theorem \ref{theorem231230}.    
\end{Proposition}

\begin{proof} The proof follows the same lines as in \cite{el2026fast}, Proposition 3.6. 
Fix $\epsilon>0$ as the target total variation distance, we prove that whenever $C_*$ is large enough, then the sampler defined by (1)-(3) above satisfies 
\begin{equation}\label{estimate1011}
    d_{TV}(\mu_{G^{(i)}},\pi^{(i)}_{k_*})\leq\frac{i\epsilon}{|V|},\quad\forall 1\leq i\leq |V|.
\end{equation}
Then taking the final $i=|V|$ so $G^{(i)}=G$ yields $d_{TV}(\mu_G,\pi^{(|V|)}_{k_*})\leq\epsilon$. This estimate will be proved inductively, and for the base case the bound is trivial as we use a perfect sample.

Assume the bound holds for $i-1$, then $\pi_0^{(i)}$ has total variation distance at most $(i-1)\epsilon/|V|$ to $\mu_{G^{(i-1)}}\otimes\mu_{v_i}$. We then couple the Glauber dynamics update process $X_k^{(i)}$ starting from $\pi_0^{(i)}$ with the update process $Y_k^{(i)}$ starting from $\mu_{G^{(i-1)}}\otimes \mu_{v_i}$ via the optimal total variation coupling on the initializations $(X_0^{(i)},Y_0^{(i)})$
and then use identity coupling to couple the process if $X_0^{(i)}=Y_0^{(i)}$. Then we have 
\begin{equation}\label{estimate1018}
    d_{TV}(X_k^{(i)},\mu_{G^{(i)}})\leq d_{TV}(Y_k^{(i)},\mu_{G^{(i)}})+d_{TV}(X_0^{(i)},Y_0^{(i)})\leq (i-1)\epsilon/|V|+d_{TV}(\mu_{G^{(i)}},Y_k^{(i)}).
\end{equation}

Now the distribution of $Y_0^{(i)}$, $\mu_{G^{(i-1)}}\otimes \mu_{v_i}$, has a Radon-Nikodym density with respect to $\mu_{G^{(i)}}$ bounded by at most $M=e^{4\beta \exp(C_\alpha)}$
since $G$ has maximal degree smaller than $\exp(C_\alpha)$. Then it satisfies the warm start condition of Lemma \ref{warmstarttheorem}. Set $n=|V|$. The case $n=1$ is exact, so assume $n\geq2$.
Each $G^{(i)}$ inherits the same growth bound, since
$B_r^{G^{(i)}}(v)\subseteq B_r^G(v)$.
Together with the uniform $\operatorname{WSM}(C)$ assumption,
this allows us to apply Theorem \ref{statementtheorem906}
with constants $K,\kappa,p,q$ independent of $i$ and $n$.
Applying that theorem with $\delta/n$ and taking a union bound,
we obtain, with $\mathbb{P}_h$-probability at least $1-\delta$,
weak conductance bounds on all $G^{(i)}$ with the common
coefficient
\[
A:=K(n/\delta)^{1/q}n^\kappa.
\]
Enlarging $K$ if necessary ensures $A^p\geq2/4^{p-1}$.
Lemma \ref{warmstarttheorem} then gives, on this event,
\begin{equation}\label{dtv1025}
d_{TV}(Y_k^{(i)},\mu_{G^{(i)}})
\leq M\left(\frac{A^{2p}\log k}{k}\right)^{1/(2p-1)},
\qquad 2\leq i\leq n,\quad k\geq4.
\end{equation}

Taking $k_*=n^{C_*}$ for a sufficiently large integer
$C_*=C_*(\epsilon,\delta,\beta,C,\alpha,C_\alpha)$ makes
the right-hand side at most $\epsilon/n$.
Together with \eqref{estimate1018}, this proves
\eqref{estimate1011}.
The algorithm uses at most $nk_*=n^{C_*+1}$ Glauber updates.
If $\epsilon$ and $\delta$ are bounded below by fixed inverse
powers of $n$, then $C_*$ can still be chosen independently
of $n$, so the running time remains polynomial.
\end{proof}

\section{Control of operator norm under WSM}
\label{section666}
This section is devoted to the proof of Theorem \ref{theorem83777}. We will essentially generalize the computations in \cite{el2026fast}, Section 4 to the graph $G$. Half of the computations only rely on monotonicity and the design of the $\operatorname{SL}$ path, so they generalize immediately to RFIM on general graphs. The other half of the computations rely on polynomial volume growth of the lattice, and we prove here that after a careful adjustment, the same computation holds on $G$ with $\alpha$-stretched exponential growth for some $\alpha\in(0,1)$.

Denote by $A_t=\operatorname{Cov}(\nu_t)$.
For a vertex $u\in V$, set $\delta_t(u,0):=1$.
For an integer $\ell\geq1$, define
\begin{equation}
    \delta_t(u,\ell):=d_{TV}(\nu_t(\sigma_u\in\cdot\mid\sigma_{\partial B_\ell(u)}=+),\nu_t(\sigma_u\in\cdot\mid\sigma_{\partial B_\ell(u)}=-)),
\end{equation}where we use the graph distance on $G$ to define the ball $B_\ell(u)$. Let $\langle;\rangle_t$ denote the Gibbs average with respect to $\nu_t$, i.e. 
$$
\langle \sigma_u;\sigma_v\rangle_t:=(\operatorname{Cov}(\nu_t))_{u,v}:=\langle\sigma_u\sigma_v\rangle_t-\langle\sigma_u\rangle_t\langle\sigma_v\rangle_t
.$$
By the FKG inequality and the spatial Markov property
(see \cite{el2026fast}, Lemma 4.1),
for every integer $1\leq\ell<d(u,v)$, we have
\[
0\leq\langle\sigma_u;\sigma_v\rangle_t
\leq\delta_t(u,\ell).
\]
The same bound holds for $\ell=0$, including when $u=v$,
since the spins take values in $\{-1,1\}$.
For every integer $p\geq1$, it follows that
\begin{equation}\label{tracemomentexpansions}
\operatorname{Tr}(A_t^p)\leq\sum_{u_1,\cdots,u_p\in V}\prod_{i=1}^p \delta_t(u_i,\ell_i)
\end{equation}
provided that, for each $i$, either $\ell_i=0$ or
$1\leq\ell_i<d(u_i,u_{i-1})$, where $u_0=u_p$.
Since each $\delta_t(u_i,\ell_i)$ is bounded by 1, we can discard some terms in the product $\prod_{i=1}^p$, which provides us with more freedom in choosing $\ell_i$. Specifically, we can prove the following:

\begin{Proposition}\label{propositiononline869}
    Let $p\geq 1,t\geq 0$, consider a subset $A\subseteq [p]$ and a tuple of vertices $u_1,\cdots,u_p\in G$. Then for any sequence of positive integers $(\ell_i)_{i\in A}$ where 
\begin{equation}\label{separationcondition}
d(u_i,u_{i-1})>\ell_i,\quad i\in A,\quad\text{ and   }d(u_i,u_j)>2(\ell_i+\ell_j),\quad i\neq j\in A
    ,\end{equation}we have the factorization
    \begin{equation}
  \label{nowwepeeloff}\mathbb{E}\left[\prod_{i\in A}\delta_t(u_i,\ell_i)\mid h\right]  \leq\prod_{i\in A}\left(\sum_{v\in B_{\ell_i}(u_i)}(1+\mathbf{1}_{v=u_i})\delta_0(v,\ell_i)\right).  
    \end{equation}Here the expectation $\mathbb{E}$ is taken with respect to the Brownian localization path $\operatorname{SL}$ only.
\end{Proposition}

\begin{proof}
This Proposition is proven in exactly the same way as in \cite{el2026fast}, Proposition 4.2. We only need to replace the $\ell_\infty$ metric ball there by the ball $B(u,\ell)$ of word metric, since the graph geodesic distance perfectly satisfies the triangle inequality. In the following, we give a sketch of proof and leave the technical details to \cite{el2026fast}, Section 4.

We first verify via FKG and the SDE evolution of the $\operatorname{SL}$ paths that (see \cite{el2026fast}, Lemma 4.4),
$$
\mathbb{E}[\delta_t(u,\ell)]\leq \delta_0(u,\ell).
$$Recall that the external field $y_t$ has the distribution $t\sigma^*+B_t$ where $\sigma^*\sim\nu_0$ and that $(B_t)$ is a standard Brownian motion. We write $X_i=\mathbb{E}[\delta_t(u_i,\ell_i)\mid\sigma^*,h]$, then $X_i$ is measurable with respect to $\{\sigma_v^*,h_v:v\in B_{\ell_i}(u_i)\}$. By assumption the balls $B_{\ell_i}(u_i)$ are disjoint, so the left hand side of \eqref{nowwepeeloff} is equal to $\nu_0(\prod_{i\in A}X_i
)$.

Now we iteratively peel off elements in $A$. Assume without loss of generality that $1\in A$, then since $d(u_1,u_i)>2\ell_1+\ell_i$, the balls $B_{2\ell_1}(u_1)$ and $B_{\ell_i}(u_i)$ are disjoint for all $i\in A,i\neq 1$ so that we can check 
$$\begin{aligned}&
\nu_0(\prod_{i\in A}X_i)-\nu_0(X_1)\cdot\nu_0(\prod_{i\in A\setminus\{1\}}X_i)\\&\leq\sup_{\tau,\tau'}(\nu_0(X_1\mid \sigma^*_{\partial B_{2\ell_1}(u_1)}=\tau)-\nu_0(X_1\mid\sigma^*_{\partial B_{2\ell_1}(u_1)}=\tau'))\cdot\nu_0(\prod_{i\in A\setminus\{1\}}X_i).
\end{aligned}$$
Applying FKG inequality, we see that the supremum term is bounded from above by 
$$
\sup_{\tau,\tau'}(\nu_0(X_1\mid \sigma^*_{\partial B_{2\ell_1}(u_1)}=\tau)-\nu_0(X_1\mid\sigma^*_{\partial B_{2\ell_1}(u_1)}=\tau'))\leq \sum_{v\in B(u_1,\ell_1)}\delta_0(v,\ell_1).
$$Iteratively peeling off other elements in $A$ completes the proof.
\end{proof}

To estimate $\mathbb{E}[\operatorname{Tr}A_t^p]$ via the expansion \eqref{tracemomentexpansions}, for any given pairs of indices $(u_1,\cdots,u_p)$ we need to determine a subset $A\subset [p]$ for which we really use the estimate $\operatorname{WSM}(C)$
centered at $u_i,i\in A$ (and upper bound the other terms not indexed by $A$ simply by 1). Also, from $(u_1,\cdots,u_p)$, we determine $\ell_1,\cdots,\ell_p$ for which the separation condition \eqref{separationcondition} holds within $A$, so that we can use Proposition \ref{propositiononline869}. More formally, we need a map $A:(u_1,\cdots,u_p)\in G^p\mapsto A(u_1,\cdots,u_p)\subset[p]$, and a map $L:(u_1,\cdots,u_p)\in G^p\mapsto (\ell_1,\cdots,\ell_p)\in\{0,1,2,\ldots\}^p$, where we take $\ell_i=(L(u_1,\cdots,u_p))_i,i\in[p]$. Then we write
$$
\Gamma(A,L):=\sum_{u_1,\cdots,u_p\in V}\prod_{i\in A(u_1,\cdots,u_p)}\left(
\sum_{v\in B_{\ell_i}(u_i)}(1+\mathbf{1}_{v=u_i})\delta_0(v,\ell_i)
\right).
$$ The implicit constraint is that $(A,L)$ should satisfy \eqref{separationcondition}. 

The main combinatorial counting result is the following:
\begin{Proposition}\label{themaincombinatorial}
    Suppose $G$ has $\alpha$-stretched-exponential growth for some $\alpha\in(0,1)$ and the RFIM measure $\mu_G$ satisfies $\operatorname{WSM}(C)$ for some $C>0$. Then we can construct two maps $A$ and $L$ as above, satisfying \eqref{separationcondition} such that, for a constant $C_0=C_0(\alpha,C_\alpha,C)$,
    \begin{equation}
        \mathbb{E}_h[\Gamma(A,L)]\leq C_0^pp!|V|.
    \end{equation}
\end{Proposition}
\begin{remark}
This proposition is the only place in this section where
$\alpha$-stretched-exponential growth is used.
We follow the multiscale selection argument of \cite{el2026fast},
with adjusted integer radii and a counting estimate adapted to
$\alpha$-stretched-exponential growth.
\end{remark}

\begin{proof}
Fix an integer $p\geq1$. We first explain the reduction to connected
graphs, taking the distance between distinct connected components
to be $\infty$.
For a tuple $(u_1,\ldots,u_p)$ meeting more than one connected
component, choose the smallest $i\in\{2,\ldots,p\}$ such that
$u_i$ and $u_{i-1}$ belong to distinct components.
Set $A=\{i\}$, $\ell_i=|V|$, and $\ell_j=0$ for $j\neq i$.
These choices satisfy \eqref{separationcondition}.
The ball $B_{\ell_i}(u_i)$ is the entire component containing $u_i$,
and $\delta_0(v,\ell_i)=0$ for every vertex $v$ in that component,
since the relevant exterior boundary is empty.
Thus this tuple contributes zero to $\Gamma(A,L)$.
For tuples contained in one component, we use the construction below
within that component. Each component inherits the same growth and
WSM bounds, so summing the resulting estimates over components
proves the result for $G$.
It therefore suffices to consider connected $G$.

Fix a tuple $(u_1,\ldots,u_p)\in V^p$ and set $r_1=0$.
For $2\leq i\leq p$, define
\[
r_i=\left\lfloor\frac1{16}\min_{j<i}d(u_i,u_j)\right\rfloor,
\qquad
J_i=\min\operatorname*{arg\,min}_{j<i}d(u_i,u_j).
\]
For $k\geq0$, let
\[
Q_k=\{i\in\{2,\ldots,p\}:2^k\leq r_i<2^{k+1}\}.
\]
If every $Q_k$ is empty, set $A=\varnothing$ and $\ell_i=0$
for all $i$.
Otherwise choose the smallest $k_*$ maximizing $2^k|Q_k|$,
set $A=Q_{k_*}$, and define
\[
\ell_i=
\left\lfloor\frac18
\min_{j\in\{i-1\}\cup(A\setminus\{i\})}d(u_i,u_j)
\right\rfloor \quad(i\in A),
\qquad
\ell_i=0\quad(i\notin A).
\]
All these choices depend only on the tuple and the graph.

We verify that
\begin{equation}\label{elli12ri}
\ell_i\geq r_i\geq1,\qquad i\in A.
\end{equation}
Indeed, for a candidate $j<i$ in the minimum defining $\ell_i$,
we have $d(u_i,u_j)\geq16r_i$.
For a candidate $j>i$, necessarily $j\in A$, and hence
$d(u_i,u_j)\geq16r_j$ and $2r_j>r_i$, since $i,j$ belong
to the same dyadic group.
Thus every distance in that minimum is at least $8r_i$,
which proves \eqref{elli12ri} after taking the integer part.
Moreover,
\[
d(u_i,u_{i-1})\geq8\ell_i>\ell_i\qquad(i\in A),
\]
and, for distinct $i,j\in A$,
\[
d(u_i,u_j)\geq8\max\{\ell_i,\ell_j\}
>2(\ell_i+\ell_j).
\]
Consequently \eqref{separationcondition} holds and the closed balls
$B_{2\ell_i}(u_i)$, $i\in A$, are pairwise disjoint.

For each $i\in A$, the random variable
\[
\sum_{v\in B_{\ell_i}(u_i)}
(1+\mathbf{1}_{v=u_i})\delta_0(v,\ell_i)
\]
depends only on the fields in $B_{2\ell_i}(u_i)$.
The independence of the field coordinates, WSM, and the volume
growth bound therefore give
\[
\begin{aligned}
\mathbb{E}_h[\Gamma(A,L)]
&\leq\sum_{u_1,\ldots,u_p\in V}
\prod_{i\in A}2C\exp(C_\alpha\ell_i^\alpha-\ell_i/C)\\
&\leq C_1^p\sum_{u_1,\ldots,u_p\in V}
\exp\left(-\frac1{2C}\sum_{i\in A}r_i\right),
\end{aligned}
\]
where empty products equal $1$, and we may take
\[
C_1=\max\left\{1,\,
2C\sup_{x\geq1}\exp(C_\alpha x^\alpha-x/(2C))\right\}<\infty.
\]

Let $\alpha'=(1+\alpha)/2\in(\alpha,1)$.
If $A=\varnothing$, all $r_i$ vanish.
Otherwise, the choice of $k_*$ implies
\[
\begin{aligned}
\sum_{i=2}^p r_i^{\alpha'}
&\leq\sum_{k\geq0}2^{-(1-\alpha')k}
\sum_{i\in Q_k}r_i\\
&\leq\frac{2}{1-2^{-(1-\alpha')}}\,
2^{k_*}|Q_{k_*}|\\
&\leq\frac{2}{1-2^{-(1-\alpha')}}\sum_{i\in A}r_i.
\end{aligned}
\]
Thus, in both cases, with
$c=(1-2^{-(1-\alpha')})/(4C)>0$,
\begin{equation}\label{thereforewegets}
\mathbb{E}_h[\Gamma(A,L)]
\leq C_1^p\sum_{u_1,\ldots,u_p\in V}
\prod_{i=2}^p e^{-c r_i^{\alpha'}}.
\end{equation}

There are at most $(p-1)!$ possible sequences
$(J_2,\ldots,J_p)$, since $J_i<i$.
Given such a sequence and nonnegative integers $r_2,\ldots,r_p$,
there are $|V|$ choices for $u_1$.
After $u_1,\ldots,u_{i-1}$ have been fixed, the relation
\[
16r_i\leq d(u_i,u_{J_i})<16(r_i+1)
\]
gives at most $\exp(C_\alpha[16(r_i+1)]^\alpha)$ choices for $u_i$.
This bound also applies when $r_i=0$.
Hence \eqref{thereforewegets} yields
\[
\mathbb{E}_h[\Gamma(A,L)]
\leq C_1^p|V|(p-1)!
\left(\sum_{r=0}^{\infty}
\exp\big(C_\alpha[16(r+1)]^\alpha-c r^{\alpha'}\big)
\right)^{p-1}.
\]
The series is finite because $\alpha'>\alpha$.
Absorbing it into a constant $C_0=C_0(\alpha,C_\alpha,C)\geq1$
proves the claimed bound $C_0^p p!|V|$, including $p=1$.
\end{proof}

Now, the proof of Theorem \ref{theorem83777} is immediate:

\begin{proof}[\proofname\ of Theorem \ref{theorem83777}]
    This immediately follows from combining \eqref{tracemomentexpansions} with Proposition \ref{propositiononline869} and Proposition \ref{themaincombinatorial}.

For the tail bound, when $R\geq eC_0$, choose the integer
$m=\lfloor R/(eC_0)\rfloor\geq1$. Since the covariance
matrix is positive semidefinite, Markov's inequality
and the moment estimate give
\[
\begin{aligned}
&\mathbb{E}_h\left[\sup_{t\geq0}
\mathbb{P}\left(\|\operatorname{Cov}(\nu_t)\|_{op}
\geq R\mid h\right)\right]\\
&\qquad\leq |V|\left(\frac{C_0m}{R}\right)^m
\leq |V|e^{-m}
\leq e\,|V|e^{-R/(eC_0)}.
\end{aligned}
\]
For $0\leq R<eC_0$, the claimed bound follows from
the trivial probability bound $1$, since
$e\,|V|e^{-R/(eC_0)}\geq1$.
\end{proof}

\section*{Funding}
The author receives a fellowship from IAS provided by the S.S. Chern Foundation for Mathematical Research Fund and the Fund for Mathematics.

\printbibliography

\end{document}